\documentclass{article}
\usepackage{amsmath,amssymb}
\usepackage{amscd}
\usepackage[a4paper,nohead,headsep=0.6cm,left=2cm,right=2cm,top=2cm,
bottom=2cm,twoside]{geometry}
\usepackage{amsthm}
\usepackage{euscript}
\usepackage{latexsym}
\usepackage[matrix,arrow,curve,cmtip]{xy}
\usepackage[cp1251]{inputenc}
\usepackage[english]{babel}
\usepackage{wrapfig}
\usepackage{mathrsfs}
\usepackage{graphicx}
\usepackage{keyval}
\usepackage{mathtools}

\newtheorem{theorem}{Theorem}[section]
\newtheorem{remark}{Remark}[section]
\newtheorem{lemma}{Lemma}[section]
\newtheorem{definition}{Definition}[section]
\newtheorem{proposition}{Proposition}[section]
\newtheorem{example}{Example}[section]
\newtheorem{corollary}{Corollary}[theorem]

\begin{document}

\title{COHOMOLOGY RINGS OF THE PLACTIC MONOID ALGEBRA \\VIA A GR\"OBNER --- SHIRSHOV BASIS}

\author{VIKTOR LOPATKIN\footnote{School of Mathematical Sciences, South China Normal University, Guangzhou, wickktor@gmail.com}}

\maketitle

\abstract{In this paper we calculate the cohomology ring $\mathrm{Ext}^*_\mathrm{\Bbbk Pl_n}(\Bbbk, \Bbbk)$ and the Hochschild cohomology ring
of the plactic monoid algebra $\Bbbk \mathrm{Pl}_n$ via the Anick resolution using a Gr\"obner -- Shirshov basis.}

\section*{Introduction}

The plactic monoid was discovered by Knuth \cite{Knu}, who used an operation given by Schensted in his study of the longest increasing subsequence of a permutation. It was named and systematically studies by  Lascoux and Sch\"utzenberger \cite{LSch}, who allowed any totally ordered alphabet in the definition. It is known that the elements of plactic monoid can be written in the canonical form, and in this form can be identified with some type of the Young tableaux. Because of its strong relations to Young tableaux, the plactic monoid has already become a classical tool in several areas of representation theory and algebraic combinatorics \cite{Fulton}.

\par Among the significant applications are: a proof of the Littlewood--Richardson rule, an algorithm which allows to decompose tensor product of representations of unitary groups, a combinatorial description of the Kostka~--- Foulkes polynomials, which arise as entries of character table of the finite linear group. The plactic monoid appeared also in: theory of modular representations of the symmetric group, quantum groups, via the representation theory of quantum enveloping algebras. It is worth mentioning that even though the combinatorics of the plactic monoid has been extensively studied, there are only a few preliminary results of the corresponding plactic algebra over a field \cite{CO}.

\par The plactic monoid was connected to the free 2-nilpotent Lie algebra (which is a subalgebra of the parastatistics algebra)
in the work \cite{LodP} by J-L. Loday and T. Popov. The connection is through the quantum deformation (in the sense of Drinfeld)
of the parastatistics algebra. But the first work where the connection between the plactic monoid algebra and parastastics algebra
(in two dimensions) was found is \cite{PopovViolette}.

\par In \cite{LPl} there was given an independent proof of uniqueness of (Robinson~--- Shensted) Knuth normal forms of elements of
(Knuth~--- Sch\"utzenberger) plactic monoid.

\par Gr\"obner bases and Gr\"obner~--- Shirshov bases were invented independently by A.I. Shirshov for ideals of free
(commutative, anti-commutative) non-associative algebras \cite{Sh1,Sh2}, free Lie algebras \cite{ShE, Sh2}
by H. Hironaka \cite{Hironaka} for ideals of the power series algebras (both formal and convergent),
and by B. Buchberger \cite{Buchberger} for ideals of the polynomial algebras.

\par Cain et al \cite{Cain} use the Schensted--Knuth normal form (the set of (semistandart) Young tableaux) to prove that the multiplication table of column words (strictly decreasing words with respect to some order on the letters) forms a finite Gr\"obner~--- Shirshov basis of the finitely generated plactic monoid. In \cite{LPl} were given new explicit formulas for the multiplication tables of row (nondecreasing word) and column words and independent proofs that the resulting sets of relations are Gr\"obner~--- Shirshov bases in row and column generators respectively.

\par The Anick resolution was obtained by David J. Anick in 1986~\cite{An}. This is a resolution for a field $\Bbbk$ considered as an $A$-module,
where $A$ is an associative augmented algebra over~$\Bbbk$. This resolution reflects the combinatorial properties of $A$ because it is based
on the Composition--Diamond Lemma~\cite{LA,Be}; i.e., Anick defined the set of $n$-chains via the leading terms of a Gr\"obner~--- Shirshov basis~\cite{Sh1,Sh2,LAch}
(Anick called it the set of obstructions), and the differentials are defined inductively via $\Bbbk$-module splitting maps, the leading terms
and the normal forms of~words.

\par Later Yuji Kobayashi~\cite{Kob} obtained the resolution for a monoid algebra presented by a complete rewriting system. He constructed an effective free acyclic
resolution of modules over the algebra of the monoid whose chains are given by paths in the graph of reductions. These chains are a particular case of chains
defined by Anick~\cite{An}, and differentials have ``Anick's spirit", i.e., the differentials are described inductively via contracting homotopy, leading terms and normal
forms. Further Philippe Malbos~\cite{Mal} constructed a free acyclic resolution in the same spirit for $R \mathscr{C}$ as a $\mathscr{C}$-bimodule
over a commutative ring $R$, where $\mathscr{C}$ is a small category endowed with a convergent presentation. The resolution is constructed with the use of
the additive Kan extension of the Anick antichains generated by a set of normal forms. This construction can be adapted to the construction of
the analogous resolution for internal monoids in a~monoidal category admitting a finite convergent presentation. Malbos also showed
(using the resolution) that if a small category admits a finite convergent presentation then its Hochschild--Mitchell homology is of finite
type in all degrees.

\par The Anick resolution can be extended to the case of operads. The correspondence technique has been developed by Vladimir Dotsenko and Anton Khoroshkin~\cite{Dot}.

\par Michael J\"ollenbeck, Volkmar Welker \cite{JW} and independently of them Emil Sc\"oldberg \cite{Sc} developed a new technique "Algebraic Discrete Morse Theory".
 In particular, this technique makes it possible to describe the differentials of the Anick resolution; in fact, we have a very useful machinery for constructing
homotopy equivalent complexes just using directed graphs. Algebraic Discrete Morse Theory is algebraic version of Forman's Discrete Morse theory \cite{For1}, \cite{For2}.
Discrete Morse theory allows to construct, starting from a (regular) $\mathrm{CW}$-complex, a new homotopy-equivalent $\mathrm{CW}$-complex with fewer cells.

\par In this paper, we use this technique (the J\"ollenbeck --- Sc\"oldberg --- Welker machinery) for calculating the cohomology ring and the Hochschild cohomology ring of the plactic monoid algebra.

\section{Preliminaries.}

Let us recall some definitions and the basic concept of Algebraic Discrete Morse theory \cite{JW}, \cite{Sc}.

\smallskip

\paragraph{Basic concept.} Let $R$ be a ring and ${C}_\bullet = (C_i,\partial_i)_{i \ge 0}$ be a chain complex of free $R$-modules $C_i$. We choose a basis
$X= \cup_{i \ge 0}X_i$ such that $C_i \cong \bigoplus_{c \in X_i}Rc$. Write the differentials $\partial_i$ with respect to the basis $X$ in the following form:
$$
\partial_i: \begin{cases} {C}_i \to C_{i-1} \\ c \mapsto \partial_i(c) = \sum\limits_{x' \in X_{i-1}}[c:c']\cdot c'. \end{cases}
$$

\smallskip

\par Given a~complex $C_\bullet$ and a~basis $X$, we construct a directed weighted graph $\Gamma({C}) = (V,E)$. The set of vertices $V$ of $\Gamma({C})$
is the basis $V=X$ and the set $E$ of weighted edges is given by the rule
$$
(c,c',[c:c']) \in E \qquad \mbox{iff} \qquad c \in X_i, c' \in X_{i-1}, \, \mbox{and} \, [c:c'] \ne 0.
$$

\smallskip

\begin{definition}
A finite subset $\mathcal{M} \subset E$ in the set of edges is called an acyclic matching if it satisfies the following three conditions:
\begin{itemize}
\item \textup{(Matching)} Each vertex $v \in V$ lies in at most one edge $e \in \mathcal{M}$.

\item \textup{(Invertibility)} For all edges $(c, c' [c:c']) \in \mathcal{M}$ the weight $[c:c']$ lies in the center $\mathrm{Z}(R)$ of the ring $R$ and is a unit in $R$.
\item \textup{(Acyclicity)} The graph $\Gamma_\mathcal{M}(V,E_\mathcal{M})$ has no directed cycles, where $E_\mathcal{M}$ is given by
$$
E_\mathcal{M} : = (E \setminus \mathcal{M}) \cup \{(c',c, [c:c']^{-1}) \, \mbox{ with } \, (c,c', [c:c']) \in \mathcal{M}\}.
$$
\end{itemize}
\end{definition}

\smallskip

\par For an acyclic matching $ \mathcal{M}$ on the graph $\Gamma(C_\bullet) = (V,E)$, we introduce the following notation, which is an adaption
of the notation introduced in \cite{For1} to our situation.

\begin{itemize}
\item  We call a vertex $c \in V$ \textit{critical} with respect to $\mathcal{M}$ if $c$ does not lie in an edge $e \in \mathcal{M}$; we write
$$
X^\mathcal{M}_i:=\{c \in X_i: c \mbox{ critical}\}
$$
for the set of all critical vertices of homological degree $i$.

\item We write $c' \le c$ if $c \in X_i$, $c' \in X_{i-1}$, and $[c:c'] \ne 0$.

\item $\mathrm{Path}(c,c')$ is the set of paths from $c$ to $c'$ in the graph $\Gamma_\mathcal{M}(C_\bullet)$.

\item The weight $\omega(p)$ of a path $p = c_1 \to \ldots \to c_r \in \mathrm{Path}(c_1, c_r)$ is given by
$$
\omega(c_1 \to \ldots \to c_r) : = \prod\limits_{i=1}^{r-1} \omega (c_i \to c_{i+1}),
$$

$$
\omega(c \to c') : = \begin{cases} - \dfrac{1}{[c:c']}, \mbox{ $c \le c',$} \\ [c:c'] , \mbox{ $c' \le c,$.} \end{cases}
$$

\item We write $\Gamma(c, c'):=\sum\limits_{p \in \mathrm{Path}(c,c')}\omega(p)$ for the sum of weights of all paths from $c$ to $c'$.

\end{itemize}

\smallskip

\begin{theorem}\cite[Theorem 2.2]{JW} The chain complex $(C_\bullet, \partial_\bullet)$ of free $R$-modules is homotopy-equivalent to the complex $(C^\mathcal{M}_\bullet, \partial_\bullet^\mathcal{M})$ which is complex of free $R$-modules and
$$
C_i^\mathcal{M} : = \bigoplus\limits_{c \in X_i^\mathcal{M}}Rc,
$$
$$
\partial_i^\mathcal{M}: \begin{cases} C^\mathcal{M}_i \to C^\mathcal{M}_{i-1} \\ c \mapsto \sum\limits_{c' \in X_{i-1}^\mathcal{M}}\Gamma(c,c')c'. \end{cases}
$$

\end{theorem}

\smallskip

\par In \cite[Appendix B, Lemma B.3]{JW}, there was constructed a contracting homotopy between the Morse complex and the original complex.
We use the same denotations.

\begin{lemma}\label{homotlemma}
Let $(C_\bullet, \partial_\bullet)$ be a complex of free $R$-modules, $\mathcal{M} \subset E$ a matching on the associated graph $\Gamma(C_\bullet) = (V,E)$,
and $(C_\bullet^\mathcal{M}, \partial_\bullet^\mathcal{M})$ the Morse complex. The following maps define a chain homotopy;
$$
\check{h}_\bullet:C_\bullet \to C_\bullet^\mathcal{M}
$$
\begin{equation}\label{h>}
X_n \ni c \mapsto h(c) = \sum\limits_{c^\mathcal{M} \in X_n^\mathcal{M}}\Gamma(c, c^\mathcal{M})c^\mathcal{M},
\end{equation}

$$
\widehat{h}_\bullet: C_\bullet^\mathcal{M} \to C_\bullet
$$
\begin{equation}\label{h<}
X_n^\mathcal{M} \ni c^\mathcal{M} \mapsto h^\mathcal{M}(c^\mathcal{M}) = \sum\limits_{c \in X_n}\Gamma(c^\mathcal{M}, c)c
\end{equation}

\end{lemma}

\paragraph{Morse matching and the Anick resolution.} Throughout this paper, $\Bbbk$ denotes any field and $\Lambda$ is an associative $\Bbbk$-algebra with unity and augmentation; i.e., a $\Bbbk$-algebra homomorphism $\varepsilon:\Lambda \to \Bbbk$. Let $X$ be a set of generators for $\Lambda$. Suppose that $\le$ is a well ordering on $X^*$, the free monoid generated by $X$. For instance, given a fixed well ordering on the letters, one may order words ``length-lexicographically'' by first ordering by length and then comparing words of the same length by checking which of them occurs earlier in the dictionary. Denote by $\Bbbk \langle X \rangle$ the free associative $\Bbbk$-algebra with unity on $X$. There is a canonical surjection $f:\Bbbk \langle X \rangle \to \Lambda$ once $X$ is chosen, in other words, we get $\Lambda \cong \Bbbk\langle X \rangle / \mathrm{ker}(f)$
\par Let $\mathrm{GSB}_\Lambda$ be a Gr\"obner--Shirshov basis for $\Lambda$. Denote by $\mathfrak{V}$ the set of the leading terms in $\mathrm{GSB}_\Lambda$
and let $\mathfrak{B}= \mathrm{Irr}(\mathrm{ker}(f))$ be the set of irreducible words (not containing the leading monomials of relations as subwords)
or $\Bbbk$-basis for $\Lambda$ (see CD-Lemma~\cite{LA,Be}). Following Anick~\cite{An}, call $\mathfrak{V}$ the set of obstructions (antichains) for $\mathfrak{B}$.
For $n \ge 1$, $\upsilon = x_{i_1} \cdots x_{i_t} \in X^*$
is an {\sl $n$-prechain} whenever there exist $a_j,b_j \in \mathbb{Z}$, $1 \le j \le n$, satisfying
\par 1. $1 = a_1 < a_2 \le b_1 < a_3 \le b_2 < \ldots < a_n \le b_{n-1}<b_n = t$ and,
\par 2. $x_{i_{a_j}} \cdots x_{i_{b_j}} \in \mathfrak{V}$ for $1 \le j \le n$.
\par An $n$-prechain $x_{i_1} \cdots x_{i_t}$ is an {\sl $n$-chain} iff the integers $\{a_j,b_j\}$ can be chosen so that
\par 3. $x_{i_1} \cdots x_{i_s}$ is not an $m$-prechain for any $s < b_m$, $1 \le m \le n$.
\par As in~\cite{An}, we say that the elements of~$X$ are $0$-chains, the elements of~$\mathfrak{V}$ are $1$-chains, and denote the set
of $n$-chains by $\mathfrak{V}^{(n)}$.

\smallskip

\par As usual, the cokernel of a $\Bbbk$-module map $\eta: \Bbbk \to \Lambda$ will be denoted as $\Lambda/\Bbbk$. For each left $\Lambda$-module $C$,
 construct the relatively free $\Lambda$-module
$$
B_n(\Lambda,C):= \Lambda \otimes_\Bbbk \underbrace{(\Lambda/\Bbbk) \otimes_\Bbbk \cdots \otimes_\Bbbk(\Lambda/\Bbbk)}_{n \mbox{ \tiny {factors} } \Lambda/\Bbbk} \otimes_\Bbbk C,
$$

\par Define right $\Lambda$-module homomorphisms $\partial_n:B_n \to B_{n-1}$ for $n > 0$ by
\begin{equation*}
\partial_n( [\lambda_1| \ldots | \lambda_n]) = \lambda_1[\lambda_2| \ldots |\lambda_n] + (-1)^n[\lambda_1| \ldots| \lambda_{n-1}]\lambda_n  +  \sum\limits_{i=1}^{n-1}(-1)^{i}   [\lambda_1| \ldots | \lambda_i\lambda_{i+1}| \ldots | \lambda_n].
\end{equation*}

\par As is well known, the chain complex $(B_\bullet(\Lambda,C), \partial_\bullet)$ is a normalized bar resolution for the left $\Lambda$-module $C$.
We assume that $C = \Bbbk$, i.e., for the $c \in \Bbbk$, $\lambda \in \Lambda$ we have $ \lambda \cdot c= \varepsilon ( \lambda c)$. Let us rewrite
the resolution $(B_\bullet(\Lambda,\Bbbk), \partial_\bullet)$ as
$$
B_0 = \Lambda, \qquad B_n  = \bigoplus\limits_{\omega_1, \ldots, \omega_n \in \mathfrak{B}_\Lambda}\Lambda[\omega_1| \ldots| \omega_n], \qquad n \ge 1
$$
with differentials
\begin{equation}\label{bd}
\partial_n( [\omega_1| \ldots | \omega_n]) =\varepsilon(\omega_1)[\omega_2| \ldots |\omega_n] + (-1)^n[\omega_1| \ldots| \omega_{n-1}] \omega_n + \sum\limits_{i=1}^{n-1}(-1)^{n-i} [\omega_1| \ldots | f( \omega_i\omega_{i+1})| \ldots | \omega_n].
\end{equation}

\smallskip

\begin{theorem}[J\"ollenbeck --- Sc\"oldberg --- Welker]\label{JSW}
For $\omega \in X^*$, let $\mathfrak{V}_{\omega,i}$ be the vertices $[\omega_1| \ldots | \omega_n]$ in $\Gamma_{B_\bullet(\Lambda,\Bbbk)}$
such that $\omega = \omega_1 \cdots \omega_n$ and $i$ is the larger integer $i \ge -1$ such that $\omega_1\cdots \omega_{i+1} \in \mathfrak{V}^i$ is an Anick $i$-chain.
Let $\mathfrak{V}_\omega = \bigcup\limits_{i \ge -1}\mathfrak{V}_{\omega,i}$.
\par Define a partial matching $\mathcal{M}_\omega$ on $(\Gamma_{B_\bullet(\Lambda,\Bbbk)})_\omega = \Gamma_{B_\bullet(\Lambda,\Bbbk)}|_{\mathfrak{V}_\omega}$
by letting $\mathcal{M}_\omega$ consist of all edges
$$
[\omega_1| \ldots| \omega'_{i+2}|\omega''_{i+2}| \ldots| \omega_n] \to [\omega_1| \ldots| \omega_{i+2}| \ldots| \omega_m]
$$
when $[\omega_1|  \ldots| \omega_m] \in \mathfrak{V}_{\omega,i}$, such that $\omega'_{i+2}\omega''_{i+2} = \omega_{i+2}$
and $[\omega_1| \ldots| \omega_{i+1}|\omega'_{i+2}] \in \mathfrak{V}^{i+1}$ is an Anick $(i+1)$-chain.

\par The set of edges $\mathcal{M} = \bigcup\limits_{\omega}\mathcal{M}_\omega$ is a Morse matching on $\Gamma_{B_\bullet(\Lambda,\Bbbk)}$,
with critical cells $X_n^\mathcal{M} = \mathfrak{V}^{n-1}$ for all $n$.

\end{theorem}

\smallskip

\par From this theorem we get the following proposition (\cite[Theorem 4.4]{JW} and \cite[Theorem 4]{Sc}).
But here we assume that $\varepsilon: \Lambda \to \Bbbk$ is arbitrary.

\begin{proposition}\label{AR}
The chain complex $(A_\bullet(\Lambda),d_\bullet)$ defined by
$$
A_n(\Lambda) = \bigoplus\limits_{v \in \mathfrak{V}^{(n-1)}}\Lambda v, \qquad d_n(v) = \sum\limits_{c' \in \mathfrak{V}^{(n-2)}} \Gamma (v,v')v'
$$
where all paths from graph $\Gamma_{B_\bullet(\Lambda, \Bbbk)}^\mathcal{M}$, is the $\Lambda$-free Anick resolution for $\Bbbk$.
\end{proposition}

\smallskip

\par Let us demonstrate how the Morse matching machinery work.

\begin{example}
Let us consider the following algebra $\Lambda = \Bbbk \langle x,y \rangle/(x^2 - y^2)$. We set $x>y$, then we have
$$
xxx \xrightarrow{(-(=))} x(xx) \to xy^2
$$
and
$$
xxx \xrightarrow{((-)=)} (xx)x \to y^2x
$$
i.e., we have to add the relation $xy^2 = y^2x$, using Buchberger~--- Shirshov's algorithm we get
$$
 xxy^2 \xrightarrow{(-(=))}x(xy^2) \to x(y^2x) \to (xy^2)x \to (y^2x)x \to y^2xx \to y^2y^2 \to y^4
$$
in other hand
$$
xxy^2 \xrightarrow{((-)=)} (xx)y^2 \to y^2y^2 \to y^4
$$

\smallskip

\par Thus we get $\mathrm{GSB}_\Lambda = \{x^2 - y^2, xy^2 - y^2x\}$ then we have,
$$
\mathfrak{V} = \left\{ x^2,xy^2 \right\}, \quad  \mathfrak{V}^{(2)} = \left\{\lefteqn{\underbracket{\phantom{xx}}}x\overbracket{xx}, \lefteqn{\overbracket{\phantom{xx}}}x\underbracket{xy^2} \right\}, \quad \mathfrak{V}^{(3)} = \left\{ \lefteqn{\lefteqn{\underbracket{\phantom{xx}}}x\overbracket{xx}}\phantom{xx}\underbracket{\phantom{x}x},  \lefteqn{\lefteqn{\underbracket{\phantom{xx}}}x\overbracket{xx}}\phantom{xx}\underbracket{\phantom{x}y^2}  \right\}, \ldots,
$$
i.e., $\mathfrak{V}^{(\ell)}\Bbbk = \mathrm{Span}_\Bbbk (x^{\ell+1},x^\ell y^2)$, $\ell \ge 0$. We will use the bar notations, i.e., we will denote the elements of the set $\mathfrak{V}^{(\ell)}$ as  $[\underbrace{x|\ldots|x}_{\ell+1}]$ and $[\underbrace{x|\ldots|x}_{\ell}|y^2]$, $\ell \ge 0$. Thus we have the following (exact) chain complex
$$
\ldots \to \Lambda \otimes_\Bbbk \mathfrak{V}^{(\ell)}\Bbbk \xrightarrow{d_\ell} \Lambda \otimes_\Bbbk \mathfrak{V}^{(\ell-1)}\Bbbk \xrightarrow{d_{\ell-1}} \ldots \xrightarrow{d_2} \Lambda \otimes_\Bbbk \mathfrak{V}\Bbbk \xrightarrow{d_1} \Lambda \otimes_\Bbbk \mathrm{Span}_\Bbbk(x,y) \xrightarrow{d_0} \Lambda \xrightarrow{\varepsilon} \Bbbk \to 0.
$$

\smallskip

\par Let us define all differentials via Morse matching machinery. We have to consider the following directed weighted graphs (see fig.\ref{e1}, fig.\ref{e2} and fig.\ref{e3}).
\begin{figure}[h!]
$$
 \xymatrix{
 [x] & [x|x] \ar@{->}[r]^{\varepsilon(x)} \ar@{->}[l]_{x} \ar@{->}[d]_{-1} & [x] \\
 & [y^2] \ar@{.>}[d]_{+1} &\\
 [y] & [y|y] \ar@/_/[u]_{-1} \ar@{->}[r]^{\varepsilon(y)} \ar@{->}[l]_{y} &[y]
 }\qquad
 \xymatrix{
& [y^2] \ar@{.>}[d]_{+1} & [x|y^2] \ar@{->}[d]_{-1} \ar@{->}[r]^{\varepsilon(y^2)} \ar@{->}[l]_{x} & [x]\\
& [y|y] \ar@/_/[u]_{-1} \ar@{->}[ld]_y \ar@{->}[d]^{\varepsilon(y)} & [y^2x] \ar@{.>}[r]_{+1} & [y|yx] \ar@/_/[l]_{-1} \ar@{->}[dl]^{y} \ar@{->}[d]^{\varepsilon(yx)} \\
[y] & [y] & [yx] \ar@{.>}[d]_{+1}  & [y]\\
&[x]& [y|x] \ar@/_/[u]_{-1} \ar@{->}[r]^{\varepsilon(x)} \ar@{->}[l]_{y} &[y]&
 }
$$
\caption{Here is shown the Morse matching, the correspondence edges are shown as dots arrows.}\label{e1}
\end{figure}
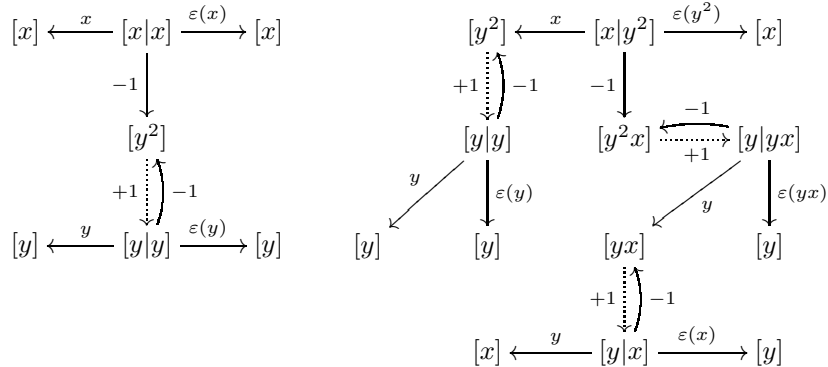

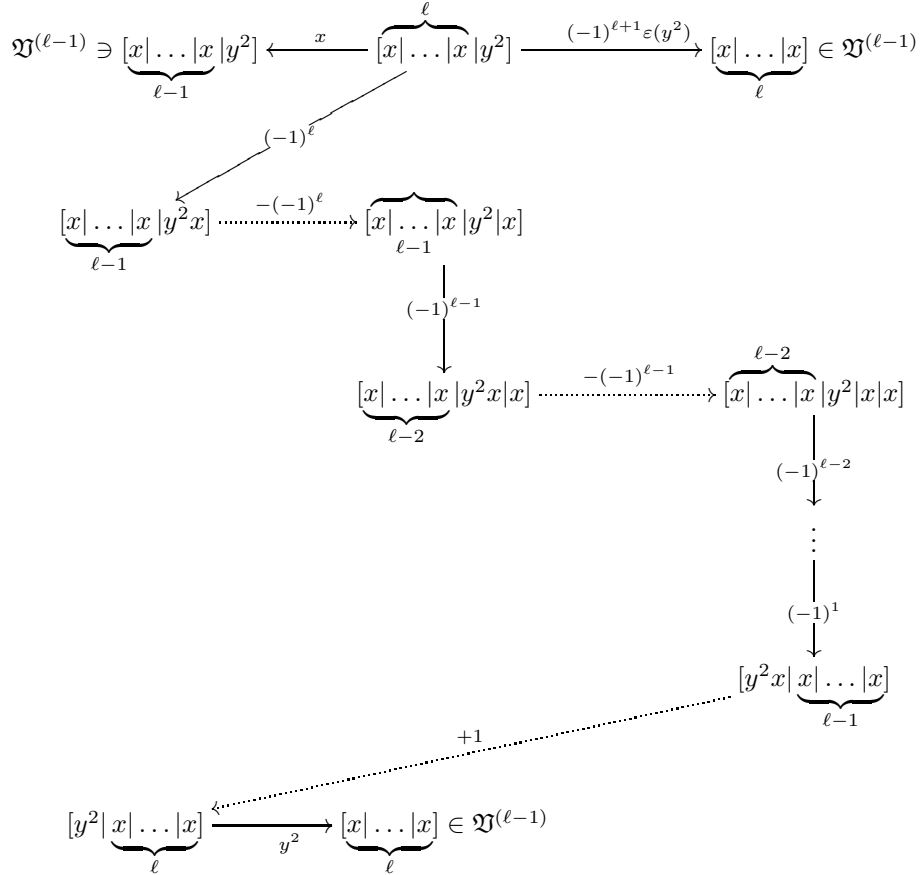
\begin{figure}[b!]
$$
  \xymatrix{
 {\mathfrak{V}^{(\ell-1)}} \ni[\underbrace{x|\ldots|x}_{\ell-1}|y^2] & [\overbrace{x|\ldots|x}^{\ell}|y^2] \ar@{->}[l]_(.4){x} \ar@{->}[rr]^{(-1)^{\ell+1}\varepsilon(y^2)} \ar@{->}[ld]|{(-1)^{\ell}} && [\underbrace{x|\ldots|x}_{\ell}]  \in {\mathfrak{V}^{(\ell-1)}} \\
  [\underbrace{x|\ldots|x}_{\ell-1}|y^2x] \ar@{.>}[r]^{-(-1)^{\ell}}& [\overbrace{x|\ldots|x}_{\ell-1}|y^2|x] \ar@{->}[d]|{(-1)^{\ell-1}} &&\\
  & [\underbrace{x|\ldots|x}_{\ell-2}|y^2x|x] \ar@{.>}[rr]^{-(-1)^{\ell-1}}&& [\overbrace{x|\ldots|x}^{\ell-2}|y^2|x|x] \ar@{->}[d]|{(-1)^{\ell-2}}\\
  &&& \vdots \ar@{->}[d]|{(-1)^1}\\
  &&& [y^2x|\underbrace{x|\ldots|x}_{\ell-1}] \ar@{.>}[llld]_{+1}\\
  [y^2|\underbrace{x|\ldots|x}_{\ell}] \ar@{->}[r]_{y^2} & [\underbrace{x|\ldots|x}_{\ell}] \in  {\mathfrak{V}^{(\ell-1)}} &&
    }
$$
\caption{As before the dots arrows mean the edges from Morse matching.} \label{e2}
\end{figure}
\begin{figure}
$$
  \xymatrix{
 {\mathfrak{V}^{(\ell-1)}} \ni [\underbrace{x|\ldots|x}_{\ell}]  & [\overbrace{x|\ldots|x}^{\ell+1}] \ar@{->}[rr]^{(-1)^{\ell+1}\varepsilon(x)}  \ar@{->}[d]|{(-1)^{i}} \ar@{->}[rd]|{(-1)^{\ell-1}} \ar@{->}[rrd]|{(-1)^{\ell}} \ar@{->}[l]_(.4){x} \ar@{->}[ld]|{(-1)^1} && [\underbrace{x|\ldots|x}_{\ell}] \in {\mathfrak{V}^{(\ell-1)}}\\
   [y^2|\underbrace{x|\ldots|x}_{\ell-1}] &[\underbrace{x|\ldots|x}_{\ell-i}|y^2|\underbrace{x|\ldots|x}_{i-1}]& [\underbrace{x|\ldots|x}_{\ell-2}|y^2|x] & [\underbrace{x|\ldots|x}_{\ell-1}|y^2] \in {\mathfrak{V}^{(\ell-1)}}
}
$$
\caption{From the previous figure follows that there is only one $(\ell-1)$th Anick's chain in bottom level.}\label{e3}
\end{figure}
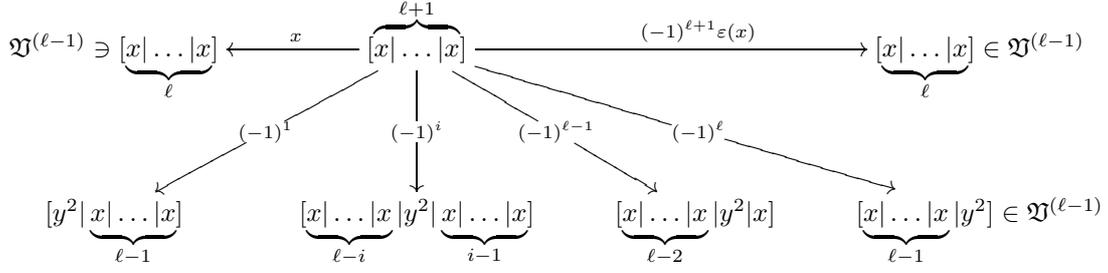

\smallskip

\par Thus we get
$$
d_1[x|x] = x[x] + \varepsilon(x)[x] - y[y] - \varepsilon(y)[y],
$$

$$
d_1[x|y^2] = \varepsilon(y^2)[x]+xy[y]+\varepsilon(y)x[y] - \varepsilon(yx)[y] - y^2[x] - \varepsilon(x)y[y],
$$

$$
d_\ell[\underbrace{x|\ldots|x}_{\ell}|y^2] = x[\underbrace{x|\ldots|x}_{\ell-1}|y^2] + (-1)^{\ell+1}\varepsilon(y^2)[\underbrace{x|\ldots|x}_{\ell}] + (-1)^\ell y[\underbrace{x|\ldots|x}_{\ell}]
$$

$$
d_\ell[\underbrace{x|\ldots|x}_{\ell+1}] = x[\underbrace{x|\ldots|x}_{\ell}] + (-1)^{\ell+1} \varepsilon(x)[\underbrace{x|\ldots|x}_{\ell}] + (-1)^\ell y^2[\underbrace{x|\ldots|x}_{\ell-1}].
$$
here $\ell >1$.
\end{example}

\smallskip 

\begin{remark}
The same algebra $\Lambda = \Bbbk \langle x,y \rangle/(x^2 - y^2)$ with nice Gr\"obner~--- Shirshov bases was considered in \cite{Ufn}. There is a small caveat; David J. Anick has developed his technique for right modules. We consider Anick's resolution for left modules. 
\end{remark}

\paragraph{Hochschild (co)homology via the Anick resolution.} Keeping the notation from the previous paragraph. As usual we denote by $\Lambda^e :=\Lambda \otimes_\Bbbk \Lambda^{\mathrm{op}}$ the enveloping algebra for algebra $\Lambda$. Follow \cite{JW}, \cite{Sc} we shall see how to construct a free $\Lambda^e$-resolution for $\Lambda$ as a (left) right module.

\par Here we consider the two-sided bar resolution $B_\bullet(\Lambda, \Lambda)$ which is an $\Lambda^e$-free resolution of $\Lambda$ where
$$
B_n(\Lambda, \Lambda):=\Lambda \otimes_\Bbbk(\Lambda/\Bbbk)^{\otimes n} \otimes_\Bbbk \Lambda \cong \Lambda^e \otimes_\Bbbk (\Lambda/\Bbbk)^{\otimes n}.
$$

\par The differential is defined as before:
\begin{equation}\label{Hd}
\partial_n([\lambda_1| \ldots | \lambda_n]) =(\lambda_1 \otimes 1) [\lambda_2| \ldots |\lambda_n] +  \sum\limits_{i=1}^{n-1}(-1)^{i}   [\lambda_1| \ldots | \lambda_i\lambda_{i+1}| \ldots | \lambda_n] + (-1)^n(1 \otimes \lambda_n)[\lambda_1| \ldots| \lambda_{n-1}] .
\end{equation}

\smallskip

\par Let us consider the same matching $\mathcal{M} = \bigcup\limits_{\omega}\mathcal{M}_\omega$ as before, and get
\begin{proposition}\cite[Chapter 5]{JW}, \cite[Lemma 9 and Theorem 5]{Sc}\label{HAR}
The set of edges $\mathcal{M} = \bigcup\limits_{\omega}\mathcal{M}_\omega$ is a Morse matching on $\Gamma_{B_\bullet(\Lambda, \Lambda)}$, with Anick chains as critical points. Moreover, the complex $(HA_\bullet(\Lambda),A\partial_\bullet)$ which is defined as follows:
$$
HA_{n+1}(\Lambda) = \Lambda^e \otimes \mathfrak{V}^{(n)}\Bbbk, \qquad A\partial_{n+1}(v) = \sum\limits_{v' \in \mathfrak{V}^{(n)}}\Gamma(v,v')v'.
$$
is a free $\Lambda^e$ resolution of $\Lambda$.
\end{proposition}

\smallskip

\par The $\Lambda^e$-resolution defined above will also be denoted by $A_\bullet(\Lambda)$. It will always be clear
from the context what kind of resolution is being considered.
\paragraph{Multiplication in Cohomology via Gr\"obner --- Shirshov basis.} Let us consider the cohomological  multiplication of associative algebra $\Lambda$ via Gr\"obner --- Shirshov basis $\mathrm{GSB}_\Lambda$. From \cite[\S 7, Chapter IX]{CE} we know that first of all we need a map
$$
g_\bullet: B_\bullet(\Lambda \otimes \Lambda) \to B_\bullet(\Lambda) \otimes B_\bullet(\Lambda)
$$
which is given by the formula
\begin{equation}
g_n[\lambda_1\otimes \lambda_1'| \ldots| \lambda_n\otimes \lambda_n'] = \sum\limits_{0 \le p \le n} [\lambda_1| \ldots| \lambda_p]\lambda_{p+1}\cdots\lambda_n \otimes \lambda_1' \cdots\lambda_p'[\lambda_{p+1}'|\ldots|\lambda_n'].
\end{equation}

\par Let us rewrite this formulae in the following way:
\begin{equation}\label{g}
g_n[\lambda\otimes \lambda']_n = \sum\limits_{0 \le p \le n} [\lambda]_{1,p}(\lambda')_{p+1,n} \otimes (\lambda)_{1,p}[\lambda']_{p+1,n},
\end{equation}
here $[\lambda \otimes \lambda']_n = [\lambda_1\otimes \lambda_1'| \ldots| \lambda_n\otimes \lambda_n'],$ $[\lambda]_{i,j}: = [\lambda_i| \ldots|\lambda_j]$, $(\lambda)_{i,j} = \lambda_i \cdots \lambda_j$, for $i \le j$ and we put that $[\lambda]_{i,j} = [\lambda]_{n+1} = [],$ $(\lambda)_{i,j} = ()$ if $i>j$. Let us consider the following diagram
$$
\xymatrix{
B_\bullet(\Lambda) \ar@{->}[r]^(.4){g_\bullet} & B_\bullet(\Lambda) \otimes B_\bullet(\Lambda) \ar@{->}[d]^{\check{h}_\bullet \otimes \check{h}_\bullet} \\
A_\bullet (\Lambda) \ar@{->}[u]^{\widehat{h}_\bullet} \ar@{->}[r]_(.4){Ag_\bullet} & A_\bullet(\Lambda) \otimes A_\bullet(\Lambda)
}
$$
from Lemma \ref{homotlemma} follows that this diagram is commutative. Then using (\ref{g}), (\ref{h>}), (\ref{h<}) we get
\begin{multline}\label{Ag}
Ag_n[\nu \otimes \nu']_n = \\ = \sum_{\substack{0 \le p \le n \\ [\lambda \otimes \lambda']_n \in (\Lambda \otimes \Lambda)^{\otimes n+1} \\ [v]_p \in \mathfrak{V}^{(p-1)}, \\ [u]_{n-p} \in \mathfrak{V}^{n-p-1}}}\Gamma([\nu \otimes \nu'],[\lambda \otimes \lambda']) \Gamma([\lambda]_{1,p},[v]_p)[v]_p(\lambda')_{p+1,n} \otimes (\lambda)_{1,p}\Gamma([\lambda']_{p+1,n},[u]_{n-p})[u]_{n-p}
\end{multline}

\par Suppose now we have a Hopf algebra $H = (H,\Delta_H,\nabla_H,\varepsilon,\eta)$ with comultiplication $\Delta_H(x) = x \otimes x$ and assume we know Gr\"obner --- Shirshov basis $\mathrm{GSB}_H$ for algebra $(H,\nabla_H,\eta)$, then for some left $H$-module $M$, we get a following commutative diagram
$$
\xymatrix{
\mathrm{Hom}_H(A_p(H), M) \otimes \mathrm{Hom}_H(A_q(H), M) \ar@{->}[r]^(.47){\bigvee}  \ar@{->}[rd]_{\smile} & \mathrm{Hom}_{H\otimes H}\left(\bigoplus\limits_{r+s = p+q}A_r(H) \otimes A_s(H), M \otimes M\right) \ar@{->}[d]^{\Delta^*_H}\\
& \mathrm{Hom}_H(A_{p+q}(H), M)
}
$$
where $\bigvee$-product \cite[\S 7, Chapter IX]{CE} is given by the following formulae,
$$
(\vartheta \bigvee \vartheta')(c \otimes c') : = \vartheta(c) \otimes \vartheta'(c'),
$$
then using (\ref{Ag}) we can describe $\smile$-multiplication by the following formulae
\begin{multline}\label{m}
(\vartheta_p \smile \vartheta_q)([v]_{p+q}) = \\ \sum_{\substack{[\lambda \otimes \lambda']_n \in (\Lambda \otimes \Lambda)^{\otimes n+1} \\ [v]_p \in \mathfrak{V}^{(p-1)}, \\ [u]_{q} \in \mathfrak{V}^{q-1}}}\Gamma([\nu \otimes \nu'],[\lambda \otimes \lambda']) \Gamma([\lambda]_{1,p},[v]_p)\vartheta_p \left([v]_p \right)(\lambda')_{p+1,p+q}(\lambda)_{1,p}\Gamma([\lambda']_{p+1,q},[u]_{q})\vartheta_q\left([u]_{q}\right)
\end{multline}

\begin{remark}
Since the comultiplication $\Delta(x) = x \otimes x$ is cocommutative, then $\vartheta_p \smile \vartheta_q = (-1)^{pq}\vartheta_q \smile \vartheta_p$, it can allow to simplify the (\ref{m}).
\end{remark}

\section{The Plactic Monoid with Column Generators}

In this section, we present an elegant algorithm proposed by C. Schensted. We will also see that there is some connection between Schensted's column algorithm and braids.

\begin{definition}
Let $A = \{1,2, \ldots, n\}$ with $1 < 2 < \cdots < n$. Then we call $\mathrm{Pl}(A) : = A^*/ \equiv$ the plactic monoid on the alphabet set $A$, where $A^*$ is the free monoid generated by $A$, $\equiv$ is the congruence of $A^*$ generated by Knuth relations $\Omega$ and $\Omega$ consists of
$$
ikj = kij \, (i \le j <k), \qquad jki = jik \,(i<j \le k).
$$
\end{definition}

\smallskip

\par For a field $\Bbbk$, denote by $\Bbbk \mathrm{Pl}(A)$ or by $\Bbbk \mathrm{Pl}_n$ the plactic monoid algebra over $\Bbbk$ of $\mathrm{Pl}(A)$.

\begin{definition}
A strictly decreasing word $w \in A^*$ is called a column~\cite{Cain}. Denote the set of columns by $\mathrm{I}$. Let $a \in \mathrm{I}$ be a column and $a_i$ the number of the letter $i$ in $a$. Then $a_i \in \{0,1\}$, $i =\{1,2, \ldots,n\}$. Put $a = (a_1; \ldots;a_n)$. Also we will consider any column as an ordered set $\{a\}:= \{a_{i_1}, \ldots, a_{i_\ell}\}$, here $\{a_{i_j}\} = \varnothing$ iff $a_{i_j} = 0$ and we will denote it by $a = e_{i_1, \ldots, i_\ell}$. Denote by $e_\varnothing$ the empty column (the unity of the plactic monoid). Also by $e_i$ we denote the column $(0;\ldots;0;1;0;\ldots;0)$ where $1$ is in $i$th place.
\end{definition}

\par For example, the word $a = 875421$ is a column, and we have $a = (1;1;0;1;1;0;1;1;0;\ldots;0)$, $\{a\} = \{a_{i_1}, a_{i_2},a_{i_3},a_{i_4},a_{i_5},a_{i_6}\}$.

\smallskip

\begin{definition}[Schensted's column algorithm]
Let $a \in \mathrm{I}$ be a column and let $x \in A$.
$$
x \cdot a = \begin{cases}xa, \, \mbox{if $xa$ is a column;} \\ a'\cdot y, \,  \mbox{otherwise}  \end{cases}
$$
where $y$ is the rightmost letter in $a$ and is larger than or equal to $x$, and $a'$ is obtained from $a$ by replacing $y$ with $x$. We say that an element $y$ is connected to $x$ or simply that  elements $y$, $x$ are connected. And we will use the notation
$$
 x \rightleftarrows y :=\begin{cases} 1, \mbox{ iff $x$ is connected to $y$,} \\ 0, \mbox{ otherwise.}  \end{cases}
$$
\end{definition}

\begin{definition}\label{v}
Consider two columns $a,b \in \mathrm{I}$ as ordered sets $\{a\}$, $\{b\}$ and consider the columns
$$
\{b^a\} : = \{x \in \{b\}: (y \rightleftarrows x) = 0 \mbox{ for any } y \in \{a\} \},
$$
$$
\{b_a\}: = \{x \in \{b\}: (y \rightleftarrows x) = 1 \mbox{ for some } y \in \{a\} \}.
$$
\par Introduce binary operations $\vee, \wedge: I \times I \to I$ by the formulas:
$$
\{a \vee b\} := \{a\} \cup \{b^a\}, \qquad \{a \wedge b\}:=\{b_a\},
$$
then from Schensted's column algorithm it follows that $a\cdot b = (a \vee b) \cdot (a \wedge b)$.
\end{definition}

\begin{figure}[h!]
\begin{picture}(110,110)
\put(70,0){
\put(100,0){
\put(0,0){\line(0,1){100}}
\put(40,0){\line(0,1){100}}
\put(0,20){\circle*{2}}
\put(-15,20){$a_{i_3}$}
\put(40,10){\circle*{2}}
\put(43,5){$b_{j_6}$}
\put(0,20){\line(4,-1){41}}
\put(0,90){\circle*{2}}
\put(-15,90){$a_{i_1}$}
\put(40,97){\circle*{2}}
\put(43,92){$b_{j_1}$}
\put(0,90){\line(3,-2){40}}
\put(40,63.33){\circle*{2}}
\put(43,64){$b_{j_2}$}
\put(0,55){\circle*{2}}
\put(-15,55){$a_{i_2}$}
\put(40,55){\circle*{2}}
\put(43,50){$b_{j_3}$}
\put(0,55){\line(1,0){40}}
\put(40,30){\circle*{2}}
\put(43,33){$b_{j_4}$}
\put(40,26){\circle*{2}}
\put(43,21){$b_{j_5}$}
\put(0,5){\circle*{2}}
\put(-15,2){$a_{i_4}$}
}
\put(200,0){
\put(0,0){\line(0,1){100}}
\put(40,0){\line(0,1){100}}
\put(0,20){\circle*{2}}
\put(-15,55){$a_{i_2}$}
\put(43,50){$b_{j_3}$}
\put(-15,20){$a_{i_3}$}
\put(40,10){\circle*{2}}
\put(43,5){$b_{j_6}$}
\put(0,20){\line(4,-1){41}}
\put(0,90){\circle*{2}}
\put(-15,85){$a_{i_1}$}
\put(0,97){\circle*{2}}
\put(-15,97){$b_{j_1}$}
\put(0,90){\line(3,-2){40}}
\put(40,63.33){\circle*{2}}
\put(43,63.33){$b_{j_2}$}
\put(0,55){\circle*{2}}
\put(40,55){\circle*{2}}
\put(0,55){\line(1,0){40}}
\put(0,30){\circle*{2}}
\put(-15,30){$b_{j_4}$}
\put(0,26){\circle*{2}}
\put(2,24){$b_{j_5}$}
\put(0,5){\circle*{2}}
\put(-15,2){$a_{i_4}$}
}
}
\end{picture}
\caption{Here is shown $a\cdot b = (a \vee b) \cdot (a \wedge b)$.}\label{ris1}
\end{figure}
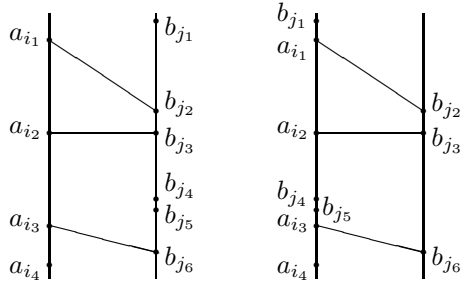

\begin{example}
Consider the following columns as ordered sets (see fig.\ref{ris1}): $\{a\} = \{a_{i_1},a_{i_2},a_{i_3},a_{i_4}\}$ and $\{b\} = \{b_{j_1},b_{j_2},b_{j_3},b_{j_4},b_{j_5},b_{j_6}\}$, we also have $a = e_{i_1, i_2,i_3,i_4}$ and $b = e_{j_1,j_2,j_3,j_4,j_5,j_6}$. We get
$$
(a_{i_1}\rightleftarrows b_{j_2}) = 1, \quad (a_{i_2}\rightleftarrows b_{j_3}) = 1, \quad (a_{i_3}\rightleftarrows b_{j_6}) = 1,
$$
then
$$
\{a \vee b\} = \{b_{j_1}, a_{i_1}, a_{i_2}, b_{j_4}, b_{j_5}, a_{i_3}, a_{i_4}\}, \qquad \{a \wedge b\} = \{ b_{j_2}, b_{j_3}, b_{j_6}\}.
$$

\end{example}

\smallskip

\begin{theorem}\label{knot}
The triple $(\mathrm{I},\vee, \wedge)$ with binary operations $\vee$ and $\wedge$ satisfies the following equations for any columns $a,b,c \in \mathrm{I}$:

\begin{equation}
a \vee b = a \vee c \mbox{ and } a \wedge b = a \wedge c \mbox{ imply } b = c,
\end{equation}

\begin{equation}
a \vee d = b \vee d \mbox{ and } a \wedge d = b \wedge d \mbox{ imply } a = b,
\end{equation}

\begin{equation}
a \vee b = a \mbox{ iff } a \wedge b = b \mbox{ and } a \wedge b = a \mbox{ iff } a \vee b = b,
\end{equation}

\begin{equation}
a \vee a = a, \quad a \wedge a = a,
\end{equation}

\begin{equation}
a \vee ( a \wedge b) = a = a \wedge (a \vee b), \quad (a \wedge b) \vee b = b = (a \vee b) \wedge b,
\end{equation}

\begin{equation}\label{black}
(a \vee b) \vee ((a \wedge b) \vee c) = a \vee (b \vee c),
\end{equation}

\begin{equation}\label{red}
(a \vee b) \wedge ((a \wedge b) \vee c) = (a \wedge (b \vee c)) \vee (b \wedge c),
\end{equation}

\begin{equation}\label{blue}
(a \wedge (b \vee c)) \wedge (b \wedge c) = (a \wedge b) \wedge c,
\end{equation}

\begin{equation}\label{comm}
a \vee b = b \Longleftrightarrow a \wedge b = a.
\end{equation}

\end{theorem}

\begin{proof}

\smallskip

\par 1. $a \vee b = a \vee c$ and $a \wedge b = a \wedge c$ imply $b = c$.
\smallskip

\par From $a \vee b = a \vee c$ it follows that $\{b^a\} =\{c^a\} $. On the other hand, from $a \wedge b = a \wedge c$ we obtain $\{b_a\} = \{c_a\}$, i.e. $b =c$.

\smallskip

\par 2.   $a \vee d = b \vee d$ and $a \wedge d = b \wedge d$ imply $a = b$.

\smallskip

\par Suppose that $\{x \vee y\} = L_x^f \cup L_x^c \cup R_y^f$. Here $L_x^f:=\{\chi \in \{x\}: (\chi \rightleftarrows u) = 0 \mbox{ for any } u \in \{y\}\}$, $L_c^x: = \{\chi\in \{x\}: (\chi \rightleftarrows u)=1 \mbox{ for some } u \in \{u\}\}$ and $R_y^f: = \{u \in \{u\}: (\chi \rightleftarrows u) = 0 \mbox{ for all } \chi \in \{x\}\}$. Consider also the set $R_y^c: = \{u \in \{u\}: (\chi \rightleftarrows u) = 1 \mbox{ for some } \chi \in \{x\}\}$. We have $\{x\} = L_x^f \cup L_x^c$ and $\{y\} = R_y^f \cup R_y^c$.

\par From $a \vee d = b \vee d$ we get $L^f_a \cup L^c_a \cup R^f_d = L^f_b \cup L^c_b \cup R'^f_d$. Further, from $a \wedge c = b \wedge c$ it follows that 2$R_d^c = R'^c_d$ but from $\{d\} = R_d^f \cup R_d^c = R'^f_d \cup R'^c_d$ we get $R^f_d = R'^f_d$, then $L^f_a \cup L^c_a \cup R^f_d = L^f_b \cup L^c_b \cup R'^f_d$
implies that $L^f_a \cup L^c_a = L^f_b \cup L^c_b $, i.e. $a = b$.

\smallskip

\par 3. $a \vee b = a$ iff $a \wedge b = b$ and $a \wedge b = a$ iff $a \vee b = b.$

\smallskip

\par From $a \vee b = a$ we conclude that $\{b^a\} = \varnothing$, and so $\{b\} = \{b_a\}$ and vice versa. If $a \wedge b = a$ then $\{b_a\} = \{a\}$, and we infer that $\{a \vee b\} = \{a\} \cup \{b^a\} = \{b_a\} \cup \{b^a\} = \{b\}$ and vice versa.

\par 5. $a \vee ( a \wedge b) = a = a \wedge (a \vee b),$ $(a \wedge b) \vee b = b = (a \vee b) \wedge b$

\smallskip

\par From $\{a \wedge b\} := \{b_a\}$ it follows that $\{(a \wedge b)^a\} = \varnothing$; i.e. $\{a \vee (a \wedge b)\} = \{a\}$. Further, $\{a \vee b\} := \{a\} \cup \{b^a\}$ yields $\{(a \wedge b)_a\} = \{a\}$.
\par Observe that $\{b_{a \wedge b}\} = \{b_a\}$ and $\{b^{a \wedge b}\} = \{b^a\}$; i.e., $\{(a \wedge b) \vee b\} = \{b\}$. Note that $\{b^{a \vee b}\} = \{b^a\} \cap \{b^{b^a}\}$ and let us prove that $\{b^{a \vee b}\} = \varnothing$. Now, $\{b^{b^a}\} = \{b_a\}$, from $\{b^a\} \cap \{b_a\} = \varnothing$ we get $\{b^{a \vee b}\} = \varnothing$,
an so $\{b_{a \vee b}\} = \{b\}$, i.e. $\{(a \vee b) \wedge b \} = \{b\}$.

\smallskip

\par 6. $(a \vee b) \wedge ((a \wedge b) \vee c) = (a \wedge (b \vee c)) \vee (b \wedge c)$.

\smallskip
\par First of all we need to describe the column $(a \wedge (b \vee c))$. Suppose that $x_\alpha \in \{a\}$, $y_\beta \in \{b\}$ and $z_\gamma \in \{c\}$ are such that $(x_\alpha \rightleftarrows y_\beta) = 1$, $(x_\alpha \rightleftarrows z_\gamma) =1$ and $(y_\beta \rightleftarrows z_\gamma) = 0$; then $\mathrm{min}\{y_\beta, z_\gamma\} \in \{a \wedge (b \vee c)\}$.

\smallskip

\par Let $x_{\gamma} \in \{(a \vee b) \wedge ((a \wedge b) \vee c)\}$. Then $x_\gamma \in \{((a \wedge b) \vee c)\}$ and there exists $y_{\beta} \in \{a \vee b\}$ such that $(y_\beta \rightleftarrows x_\gamma) = 1$. If $x_\gamma \in \{a \wedge b\}$ then there's no $z_\rho \in \{c^{a \wedge b}\}$ such that $\beta \le \gamma \le \rho$ then $x_\gamma \in \{a \wedge (b \vee c)\}$, i.e., $x_\gamma \in \{(a \wedge (b \vee c)) \vee (b \wedge c)\}$. Let $x_\gamma \in \{c^{(a \wedge b)}\}$ then there is no $x_\rho \in \{a \wedge b\}$ such that $\beta \le \rho \le \gamma$,
and hence  $x_\gamma \in \{a \wedge (b \vee c)\}$, i.e., $x_\gamma \in \{(a \wedge (b \vee c)) \vee (b \wedge c)\}$. We have proved that $\{(a \vee b) \wedge ((a \wedge b) \vee c)\} \subseteq \{(a \wedge (b \vee c)) \vee (b \wedge c)\}$.

\par Let $x_\gamma \in \{(a \wedge (b \vee c)) \vee (b \wedge c)\}$. If $x_\gamma \in \{(a \wedge (b \vee c))\}$ then $x_\gamma \in \{b \vee c\}$ and there exists $y_\alpha \in \{a\}$ with $(y_\alpha \rightleftarrows x_\gamma) = 1$. Assume that $x_\gamma \in \{b\}$. Then there is no $z_\beta \in \{c\}$ such that $\alpha \le \gamma \le \beta$. Therefore, $x_\gamma \in \{(a \wedge b) \vee c\}$. Since $y_\alpha \in \{a\}$, it follows that $y_\alpha \in \{a \vee b\}$ and, since $(y_\alpha \rightleftarrows x_\gamma) = 1$, we see that $x_\gamma \in \{(a \vee b) \wedge ((a \wedge b) \vee c)\}$. Suppose now that $x_\gamma \in \{c^b\}$. Then there is no $z_\rho \in \{b\}$ with $\alpha \le \rho \le \gamma$, and hence $x_\gamma \in \{(a \wedge b) \vee c\}$, and, since there is $y_\alpha \in \{a\}$, we get $x_\gamma \in \{(a \vee b) \wedge ((a \wedge b) \vee c)\}$. Now, consider the case $x_\gamma \in \{(b \wedge c)\}$. There exists $z_\beta \in \{b\}$ such that $(z_\beta \rightleftarrows x_\gamma) = 1$, and there is no $y_\alpha \in \{a \wedge (b \vee c)\}$ such that $(y_\alpha \rightleftarrows x_\gamma) =1$, i.e., if for some $u_\rho \in \{a\}$ there exists $x_{\gamma'} \in \{c\}$ with $(u_\rho \rightleftarrows x_{\gamma'}) =1$ then $\alpha \le \gamma'<\beta$; i.e., $z_\beta \in \{b^a\}$, and hence $x_\gamma \in \{(a \vee b) \wedge ((a \wedge b)\vee c)\}$. We have proved that $\{(a \vee b) \wedge ((a \wedge b) \vee c)\} \supseteq \{(a \wedge (b \vee c)) \vee (b \wedge c)\}$.

\smallskip

\par 7. $(a \wedge (b \vee c)) \wedge (b \wedge c) = (a \wedge b) \wedge c.$

\smallskip

\par Take $c_{\gamma_1} \in \{c\}$ and $c_{\gamma_1} \in \{(a \wedge (b \vee c)) \wedge (b \wedge c)\}$. Then there exists $b_{\beta_1} \in \{b\}$ with $(b_{\beta_1} \rightleftarrows c_{\gamma_1})=1$ and also there exists $x_{c_{\gamma_1}} \in \{a \wedge (b \vee c)\}$ such that $(x_{c_{\gamma_1}} \leftrightarrows c_{\gamma_1}) = 1$ and $(x_{c_{\gamma_1}} \rightleftarrows b_{\beta_1}) = 1$. Since $x_{c_{\gamma_1}} \in \{a \wedge (b \vee c)\}$, there exists $y_\alpha \in \{a\}$ with  $(y_\alpha \rightleftarrows x_{c_{\gamma_1}}) = 1$.

\par Observe that if $x_{c_{\gamma_1}} \in \{b\}$ then $(x_{c_{\gamma_1}}\rightleftarrows c_{\gamma_1}) = 1$ and $(b_{\beta_1}\rightleftarrows c_{\gamma_1}) = 1$ but it is possible iff $x_{c_{\gamma_1}} = b_{\beta_1}$. This means that $c_{\gamma_1} \in \{c\}$ is connected with some $x_{c_{\gamma_1}} \in \{b\}$ which is connected with some $y_\alpha \in \{a\}$; i.e., $\{(a \wedge (b \vee c)) \wedge (b \wedge c)\} \subseteq \{ (a \wedge b) \wedge c\}$.

\par If $x_{c_{\gamma_1}} \in \{c^b\}$ then $\gamma_1 \le \beta_1$ because $(x_{c_{\gamma_1}} \rightleftarrows b_{\beta_1})=1$. This means there is no $b_{\beta_2} \in \{b\}$ with $\alpha_1 \le \beta_2 \le \gamma_1$; i.e., $(b_{\beta_1} \rightleftarrows y_{\alpha_1}) =1$, and hence $(y_{\alpha_1} \rightleftarrows b_{\beta_1})(b_{\beta_1} \rightleftarrows c_{\gamma_1}) =1$; i.e.,  $\{(a \wedge (b \vee c)) \wedge (b \wedge c)\} \subseteq \{ (a \wedge b) \wedge c\}$.

\par Let $c_{\gamma_1} \in \{(a \wedge b) \wedge c\}$, i.e., $c_{\gamma_1} \in \{c\}$ and there exists $b_{\beta_1} \in \{a \wedge b\}$ such that $(b_{\beta_1}\rightleftarrows c_{\gamma_1})=1$. Also for $b_{\beta_1}$ there exists $a_{\alpha_1} \in \{a\}$ such that $(a_{\alpha_1}\rightleftarrows b_{\beta_1}) = 1$. Then $c_{\gamma_1} \in \{b \wedge c\}$, and since $(b_{\beta_1}\rightleftarrows c_{\gamma_1})=1$, we may assume that $a_{\alpha_1} \in \{a \wedge (b \vee c)\}$, i.e., $\{(a \wedge b) \wedge c\} \subseteq \{(a \wedge (b \vee c)) \wedge (b \wedge c)\}$

\smallskip

\par 8. $(a \vee b) \vee ((a \wedge b) \vee c) = a \vee (b \vee c)$.

\smallskip

\par Is not hard to see that $\{a \vee (b \vee c)\} = \{a\} \cup \{(b \vee c)^a\} = \{a\} \cup \{b^a\} \cup \{(c^b)^a\}$ because $b \wedge c^b = e_\varnothing$. Then we get $\{(a \vee b) \vee ((a \wedge b)\vee c)\} = \{a\} \cup \{b^a\} \cup \{((a \wedge b) \vee c)^{(a \vee b)}\} = \{a\} \cup \{b^a\} \cup \{(a\wedge b)^{(a \vee b)}\} \cup \{c^{(a \vee b)}\}  = \{a\} \cup \{b^a\} \cup \{c^{(a \vee b)}\}$, but $\{c^{(a \vee b)}\} = \{(c^b)^a\}$; i.e., $\{(a \vee b) \vee ((a \wedge b) \vee c)\} = \{a \vee (b \vee c)\}$; as claimed.

\smallskip

\par 9.
$$\label{com}
a \vee b = b \Longleftrightarrow a \wedge b = a
$$

\smallskip

\par Indeed, since $\{a \vee b\}: = \{a\} \cup \{b^a\}$ and $\{b\} = \{b_a\} \cup \{b^a\}$, it follows from $a \vee b = b$ that $\{a\} = \{b_a\}$.

\end{proof}

\smallskip

\par Suppose that $a=(a_1; \ldots; a_n) \in \mathrm{I}$ and $\mathrm{wt}(a): = (a_1+\ldots+a_n,a_1, \ldots, a_n)$. Order $\mathrm{I}$ as follows: for any $a,b \in \mathrm{I}$, we say that $a < b$ whenever $\mathrm{wt}(a) > \mathrm{wt}(b)$ lexicographically. Then order $\mathrm{I}^*$ by the deg-lex order.

\smallskip

\begin{remark} Let $a, b \in \mathrm{I}$. Then Schensted's column algorithm and Definition \ref{v} imply that $a\cdot b$ is the leading term iff $a \vee b \ne a$ and $a \wedge b \ne b$.
\end{remark}

\begin{remark} Put $\mathcal{I}: = \{a\cdot b = (a \vee b) \cdot (a \wedge b):a,b \in \mathrm{I}\}$. Then we may assume \cite{LPl} that $\Bbbk \langle \mathrm{I}\rangle/ (\mathcal{I}) \cong \Bbbk \langle A \rangle/ (\Omega)$. Then formulas (3.6), (3.7), (3.8) enable us to prove that $\mathcal{I}$ is the Gr\"obner --- Shirshov basis of the plactic monoid in column generators (see also \cite[Theorem 4.3]{LPl}). In fig. \ref{braidPl}, we show a sketch of the Buchberger --- Shirshov algorithm for the plactic monoid via the binary operations $\vee$ and $\wedge$. In the knot theory spirit,  we can interpret this operation as ``overcrossing'' and ``udercrossing''.
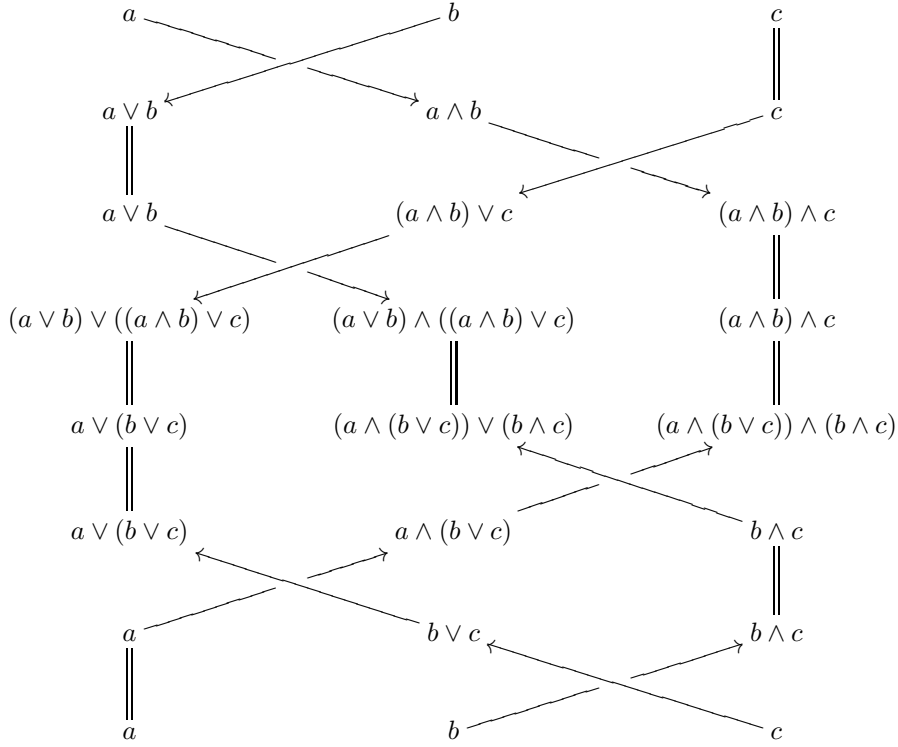
\begin{figure}[h!]
 $$
   \xymatrix{
   a \ar@{->}[rd]|{\phantom{AA}} & b \ar@{->}[ld] & c \ar@{=}[d] \\
   a \vee b \ar@{=}[d] & a\wedge b \ar@{->}[rd]|{\phantom{AA}} & c \ar@{->}[ld] \\
   a \vee b \ar@{->}[rd]|{\phantom{AA}} & (a \wedge b) \vee c \ar@{->}[ld] & (a \wedge b) \wedge c \ar@{=}[d] \\
   (a \vee b) \vee ((a \wedge b) \vee c) \ar@{=}[d] & (a \vee b) \wedge ((a \wedge b) \vee c) \ar@{=}[d] &  (a \wedge b) \wedge c \ar@{=}[d]\\
   a \vee (b \vee c) & (a \wedge (b \vee c)) \vee (b \wedge c) & (a \wedge (b \vee c))\wedge (b \wedge c)\\
   a \vee (b \vee c) \ar@{=}[u] & a \wedge (b \vee c) \ar@{->}[ur]|{\phantom{AA}} & b \wedge c \ar@{->}[ul] \\
   a \ar@{->}[ur]|{\phantom{AA}} & b \vee c \ar@{->}[ul] & b \wedge c \ar@{=}[u]\\
   a \ar@{=}[u] & b \ar@{->}[ur]|{\phantom{AA}} & c \ar@{->}[ul]
   }
 $$
\caption{The ``braid diagram'' for the Buchberger --- Shirshov algorithm for the plactic monoid via column generators.}\label{braidPl}
\end{figure}
\end{remark}

\begin{remark}
The operations $\vee$, $\wedge$ are not associative. Indeed, suppose that $a = e_{i},$ $b = e_j$ and $c = e_k$. Assume that $j < k < i$. Then $(e_i \vee e_j) \vee e_k = e_{ji} \vee e_{k} = e_{ji}e_k$. But $e_i \vee (e_j \vee e_k) = e_i \vee e_k = e_{ik}$; i.e.,
$$
(a \vee b) \vee c \ne a \vee (b \vee c).
$$
\par Assume now that $j<i<k$. Then $(e_i \wedge e_j)\wedge e_k = e_\varnothing \wedge e_k = e_\varnothing$. On other hand, $e_i \wedge (e_j \wedge e_k) = e_i \wedge e_k = e_k$; i.e.,
$$
(a \wedge b) \wedge c \ne a \wedge (b \wedge c).
$$
\par However, we will use the notation $a \vee (b \vee c) := a \vee b \vee c$ and $(a \wedge b) \wedge c:=a \wedge b \wedge c$.
\end{remark}

\section{The Anick Resolution via Column Generators}
Here we describe the Anick resolution for the $\Bbbk\mathrm{Pl}_n$-module $\Bbbk$ and for the  $\Bbbk \mathrm{Pl}_n^e  = \Bbbk \mathrm{Pl}_n \otimes_\Bbbk \Bbbk \mathrm{Pl}_n^{\circ}$-module $\Bbbk \mathrm{Pl}_n$.

\smallskip

\begin{lemma}\label{twoholes}
Given four arbitrary letters (columns) $a,b,c,d$, consider the word $abcd$ and suppose that $a \wedge b = b$, $b \wedge c \ne c$, and $c \wedge d =d$.  Then there is no reduction of this word to a word of the form $a'b'c'd'$ such that $a' \wedge b' \ne b'$, $b'\wedge c' \ne c'$ or
$b'\wedge c' \ne c'$, $c' \wedge d' \ne d'$.
\end{lemma}

\begin{proof}
We have
$$
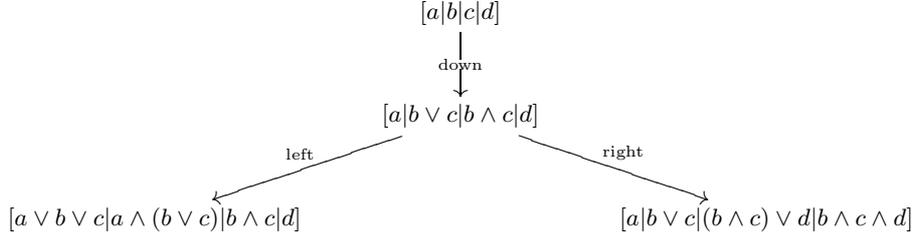
\xymatrix{
& [a|b|c|d] \ar@{->}[d]|{\mathrm{down}} & \\
&[a| b \vee c| b \wedge c| d] \ar@{->}[rd]^{\mathrm{right}} \ar@{->}[ld]_{{\mathrm{left}}} &\\
[a \vee b \vee c| a \wedge (b \vee c)|b \wedge c| d] && [a| b \vee c| (b \wedge c) \vee d| b \wedge c \wedge d]
}
$$
\par 1. Consider the lower left side. Since $a \vee b = a$, $a \wedge b =b$, (\ref{red}) implies that
$$
(a \wedge (b \vee c)) \vee (b \wedge c) = (a \vee b) \wedge ((a \wedge b) \vee c) = a \wedge (b \vee c);
$$
i.e., $(a \wedge (b \vee c)) \cdot (b \wedge c)$ is not the leading term.

\par 2. Consider the  lower right side. Since $c \vee d = c$, $c \wedge d = d$, (\ref{red}) implies that
$$
(b \vee c) \wedge ((b \wedge c) \vee d ) = (b \wedge (c \vee d)) \vee (c \wedge d) = (b \wedge c) \vee d;
$$
i.e., $(b \vee c) \cdot ((b \wedge c) \vee d )$ is not the leading term.
\end{proof}

\smallskip

\begin{lemma}\label{keepchainlemma}
Let $a$, $b$, and $c$ be columns such that $a \wedge b \ne b$ and $b \wedge c \ne c$. Then $a \wedge (b \vee c) \ne b \vee c$ and $(a \wedge b) \wedge c \ne c$.
\end{lemma}

\begin{proof}
Theorem \ref{knot} implies that
$$
a \wedge (b \vee c) \ne a \wedge b \ne b \ne b \vee c, \qquad (a \wedge b) \wedge c \ne b\wedge c \ne c,
$$
as claimed.
\end{proof}

\smallskip

\begin{theorem}\label{CPl}
Let $\Bbbk\mathrm{Pl}(A)$ be the plactic monoid algebra over the field $\Bbbk$ with augmentation $\varepsilon:\Bbbk\mathrm{Pl}(A) \to \Bbbk$, and let $\mathrm{I}$ be a set of generators (columns) of the plactic monoid. Then the vector space $\mathfrak{V}^{(m)}\Bbbk$ spanned by the vectors $a_1 \cdots a_{m+1}$ such that $a_i\wedge a_{i+1} \ne a_{i+1}$ for all $1 \le i \le m$ form an $m$-Anick chain;
moreover, there is an (exact) chain complex of $\Bbbk \mathrm{Pl}_n$-modules:
$$
0 \xleftarrow{} \Bbbk \xleftarrow{\varepsilon} \Bbbk \mathrm{Pl}_n \xleftarrow{d_1} \Bbbk \mathrm{Pl}_n \otimes_\Bbbk \mathrm{I}\Bbbk \xleftarrow{d_2} \Bbbk \mathrm{Pl}_n \otimes_\Bbbk\mathfrak{V}\Bbbk \xleftarrow{d_3} \Bbbk \mathrm{Pl}_n \otimes_\Bbbk\mathfrak{V}^{(2)}\Bbbk \xleftarrow{} \ldots,
$$
where
\begin{multline}\label{AdPl}
d_n([a_1| \ldots | a_\ell]) = \\ =\sum\limits_{i=0}^{\ell-1}(-1)^{i} (a_1 \vee \ldots  \vee a_{i+1}) [\widehat {L_i}] + \sum\limits_{j=1}^{\ell} (-1)^{j}\varepsilon(a_i \wedge \ldots \wedge a_\ell)[\widehat{R_i}] + \sum\limits_{m=1}^{\ell-1}\sum\limits_{m+l+k \le \ell-1}(-1)^{l+k}W_{m,l,k}.
\end{multline}
Here, for $1 \le i,j \le \ell -1$,
\begin{equation}\label{Li}
[\widehat{L_i}]  =\begin{cases}0, \mbox{iff } a_j \vee a_{j+1} = a_j \vee (a_{j+1} \vee a_{j+2}) \mbox{ for some } i \le j \le \ell-2, \\ [a_1 \wedge (a_2 \vee \cdots \vee a_{i+1})| \ldots | a_{i-1}\wedge (a_i \vee a_{i+1})| a_i \wedge a_{i+1}| a_{i+2}| \ldots | a_\ell], \mbox{ otherwise.} \end{cases}
\end{equation}

\begin{equation}\label{Rj}
[\widehat{R_i}]  = \begin{cases} 0, \mbox{ iff } a_j \wedge a_{j+1} = (a_j \wedge a_{j+1})\wedge a_{j+2} \\ [a_1| \ldots | a_{i-1}| a_i \vee a_{i+1}|(a_i \wedge a_{i+1})\vee a_{i+2}|  \ldots| (a_i \wedge \cdots \wedge a_{\ell-1}) \vee a_n], \mbox{ otherwise.} \end{cases}
\end{equation}
Here $\widehat{L_0} = [a_2|\ldots|a_\ell]$, $\widehat{R_{\ell}} = [a_1|\ldots|a_{\ell}],$
\begin{multline*}
W_{m,l,k}  = \\ = \begin{cases} [a_1| \ldots |a_{m-1}|b_m|\ldots|b_{m+l-1}|b_{m+l} \vee c_{m+l+1}|c_{m+l+2}|\ldots|c_{m+l+k}|a_{m+l+k+1}|\ldots|a_\ell], \mbox{ iff } b_{m+l} \wedge c_{m+l+1} = 1_{\mathrm{Pl}_n} \\ 0, \mbox{ otherwise,} \end{cases}
\end{multline*}
here
\begin{multline*}
b_{m+\iota} = \begin{cases} a_m \vee a_{m+1}, \mbox{ if } \iota =0 \\ (a_m \wedge \cdots \wedge a_{m+\iota})\vee a_{m+\iota+1}, \mbox{ if } 1 \le \iota \le l-1 \\ a_m \wedge \cdots \wedge a_{m+l}, \mbox{ if } \iota = l \end{cases} \mbox{ or } \\ b_{m+\iota} =- \begin{cases}a_m \vee \cdots \vee a_{m+l+1}, \mbox{ if } \iota = 0 \\ a_{m+\iota-1} \wedge (a_{m+\iota} \vee \cdots \vee a_{m+l}), \mbox{ if } 1 \le \iota \le l-1\\ a_{m+l-1} \wedge a_{m+l}, \mbox{ if } \iota = l  \end{cases}
\end{multline*}
\begin{multline*}
c_{m+\nu} = -\begin{cases}a_{m+l} \vee \cdots \vee a_{m+l+k}, \mbox{ if } \nu=l+1\\   a_{m+\nu-1} \wedge (a_{m+\nu}\vee \cdots \vee a_{m+l+k}), \mbox{ if } \nu = l+t, \, 1 \le t < k \\ a_{m+l+k-1} \wedge a_{m+l+k}, \mbox{ if } \nu = l+k  \end{cases} \mbox{ or } \\ c_{m+\nu} = \begin{cases} a_{m+l+1} \vee a_{m+l+2}, \mbox{ if } \nu= l+1 \\ (a_{m+l+1} \wedge \cdots \wedge a_{m+\nu})\vee a_{m+\nu+1}, \mbox{ if } \nu = l+t, \, 1 < t< k \\ a_{m+l+1} \wedge \cdots \wedge a_{m+l+k}, \mbox{ if } \nu = l+k. \end{cases}
\end{multline*}

\end{theorem}
\begin{proof} Since the Gr\"obner --- Shirshov basis of the plactic monoid via column generators is quadratic nonhomogeneous, $\mathfrak{V}^{(m)} = \{a_1\cdots a_{m+1}: a_i\wedge a_{i+1} \ne a_{i+1}, \mbox{ for all } 1 \le i \le m \}$, for any $m>1$. Following \cite{JW}, we will use the bar notation $[a_1|\ldots|a_{\ell+1}]$ for an $\ell$th Anick chain.

\smallskip

\par Let $[a_1| \ldots| a_\ell] \in \mathfrak{V}^{(\ell-1)}$ be an $\ell-1$-Anick chain. Theorem \ref{JSW} and Proposition \ref{AR} tell us that first we must find all weighted paths $p_i: [a_1|\ldots|a_\ell] \xrightarrow{\omega_i}[b_1|\ldots|b_{\ell-1}]$ such that $[b_1|\ldots|b_{\ell-1}] \in \mathfrak{V}^{(\ell-2)}$.

\par We say that the $n$-tuple $[a_1|\ldots|a_i \vee a_{i+1}|a_i \wedge a_{i+1}| \ldots|a_\ell]$ has a hole at the point $i$. Lemma \ref{keepchainlemma} implies that we can move this hole to the left or to the right in the following sense:
\begin{multline*}
[a_1|\ldots|a_i \vee a_{i+1}|a_i \wedge a_{i+1}| \ldots|a_\ell] \to \\ \to  [a_1|\ldots| a_i\vee a_{i+1}| (a_i \wedge a_{i+1})\vee a_{i+2}|(a_i \wedge a_{i+1})\wedge a_{i+2}|\ldots|a_\ell] \quad \mbox{ movement of the hole to the right by one step}
\end{multline*}
\begin{multline*}
\mbox{movement of the hole to the left by one step } \quad [a_1|\ldots|a_i \vee a_{i+1}|a_i \wedge a_{i+1}| \ldots|a_\ell] \to \\ \to [a_1|\ldots| a_{i-1} \vee (a_i \vee a_{i+1})|a_{i-1} \wedge (a_i \vee a_{i+1})|a_i \wedge a_{i+1}| \ldots|a_\ell].
\end{multline*}

\par Lemma \ref{twoholes} implies that, for finding paths $p_i: [a_1|\ldots|a_\ell] \xrightarrow{\omega_i}[b_1|\ldots|b_{\ell-1}]$, where $[b_1|\ldots|b_{\ell-1}] \in \mathfrak{V}^{(\ell-2)}$, we cannot make more than one hole in the tuple $[a_1|\ldots|a_\ell]$. Assume that  $(a_i \vee a_{i+1})\vee ((a_i \wedge a_{i+1}) \vee a_{i+2}) \ne a_i \vee a_{i+1}$ and $(a_i \wedge (a_{i+1} \vee a_{i+2}))\wedge (a_{i+1} \wedge a_{i+2}) \ne a_{i+1} \wedge a_{i+2}$ for any $1 \le i \le \ell-2$
Then all paths $p_i$ have the form
\begin{multline*}
L_i :[a_1|\ldots|a_\ell] \to \\ \to  [a_1 \vee \cdots \vee a_{i+1}| a_1 \wedge (a_2 \vee \cdots \vee a_{i+1})| a_2 \wedge (a_3 \vee \cdots \vee a_{i+1})| \ldots | a_{i-1}\wedge (a_i \vee a_{i+1})| a_i \wedge a_{i+1}| a_{i+2}| \ldots | a_\ell]\,,
\end{multline*}
\begin{multline*}
R_i :[a_1|\ldots|a_\ell] \to \\ \to [a_1| \ldots | a_{i-1}| a_i \vee a_{i+1}|(a_i \wedge a_{i+1})\vee a_{i+2}| (a_i \wedge a_{i+1} \wedge a_{i+2}) \vee a_{i+3}| \ldots| (a_i \wedge \cdots a_{n-1}) \vee a_\ell| a_1 \wedge \cdots \wedge a_\ell].
\end{multline*}

\par Since $\Gamma([a_1|\ldots|a_\ell] \to L_i) = (-1)^i$ and $\Gamma([a_1|\ldots|a_\ell] \to R_i) = (-1)^{\ell -i}$, (\ref{bd}) implies
$$
\Gamma([a_1|\ldots|a_\ell] \to L_i \to \widehat{L_i}) = (-1)^i\varepsilon(a_1 \vee \ldots \vee a_{i+1}), \quad \Gamma([a_1|\ldots|a_\ell] \to R_i\to \widehat{R_i}) =(-1)^{i}(a_i \wedge \ldots \wedge a_\ell),
$$
and Proposition~\ref{AR} yields~(\ref{AdPl}).

\smallskip

\par Now, suppose that, for some $0 \le i \le \ell$, we have $(a_i \vee a_{i+1})\wedge ((a_i \wedge a_{i+1}) \vee a_{i+2}) = ((a_i \wedge a_{i+1}) \vee a_{i+2})$ or $(a_i \wedge (a_{i+1} \vee a_{i+2}))\wedge (a_{i+1} \wedge a_{i+2}) = a_{i+1} \wedge a_{i+2}$. Theorem \ref{knot} implies that
$$
a_i \vee a_{i+1} = (a_i \vee a_{i+1})\vee ((a_i \wedge a_{i+1}) \vee a_{i+2}) = a_i \vee (a_{i+1} \vee a_{i+2}),
$$
$$
a_{i+1} \wedge a_{i+2} = (a_i \wedge (a_{i+1} \vee a_{i+2}))\wedge (a_{i+1} \wedge a_{i+2}) = (a_i \wedge a_{i+1})\wedge a_{i+2},
$$
and we must put $\widehat{R_i} = 0$ (respectively, $\widehat{L_i} =0$). Finally, assuming that there are equalities of the form $a' \wedge b' = 1_{\mathrm{Pl_n}}$, we infer that there are paths of the form $W_{m,l,k}$ with weight $(-1)^{l+k}$. This completes the proof.

\end{proof}

\smallskip

\begin{theorem}\label{HARPl} In the above notation, we obtain the (exact) chain complex of $\Bbbk \mathrm{Pl}_n^e$-modules
$$
0 \xleftarrow{} \Bbbk \mathrm{Pl}_n^e \xleftarrow{d_0} \Bbbk \mathrm{Pl}_n^e \otimes_\Bbbk \mathrm{I}\Bbbk \xleftarrow{d_1} \Bbbk \mathrm{Pl}_n^e \otimes_\Bbbk\mathfrak{V}\Bbbk \xleftarrow{d_2} \Bbbk \mathrm{Pl}_n^e \otimes_\Bbbk\mathfrak{V}^{(2)}\Bbbk \xleftarrow{} \ldots,
$$
where
\begin{multline}\label{HAdPl}
d_n([a_1| \ldots | a_\ell]) = \\ = \sum\limits_{i=0}^{\ell-1}(-1)^{i} ((a_1 \vee \ldots  \vee a_{i+1}) \otimes 1) [\widehat {L_i}] + \sum\limits_{j=1}^\ell(-1)^{j} (1 \otimes (a_j \wedge \ldots \wedge a_\ell)) [\widehat{R_j}] +  \sum\limits_{m=1}^{\ell-1}\sum\limits_{m+l+k \le \ell-1}(-1)^{l+k}W_{m,l,k}.
\end{multline}
\end{theorem}
\begin{proof}
The proof is the same as that of Theorem \ref{CPl} with the exception of weights. As in the proof of Theorem \ref{CPl}, we can give an explicit description of the paths:
\begin{multline*}
L_i :[a_1|\ldots|a_\ell] \to \\ \to  [a_1 \vee \cdots \vee a_{i+1}| a_1 \wedge (a_2 \vee \cdots \vee a_{i+1})| a_2 \wedge (a_3 \vee \cdots \vee a_{i+1})| \ldots | a_{i-1}\wedge (a_i \vee a_{i+1})| a_i \wedge a_{i+1}| a_{i+2}| \ldots | a_\ell]\,,
\end{multline*}
\begin{multline*}
R_i :[a_1|\ldots|a_\ell] \to \\ \to [a_1| \ldots | a_{i-1}| a_i \vee a_{i+1}|(a_i \wedge a_{i+1})\vee a_{i+2}| (a_i \wedge a_{i+1} \wedge a_{i+2}) \vee a_{i+3}| \ldots| (a_i \wedge \cdots a_{n-1}) \vee a_\ell| a_1 \wedge \cdots \wedge a_\ell].
\end{multline*}

\par It follows from (\ref{Hd})
$$
\Gamma([a_1|\ldots|a_\ell] \to L_i \to \widehat{L_i}) = (-1)^i(a_1 \vee \ldots \vee a_{i+1})\otimes 1, \quad \Gamma([a_1|\ldots|a_\ell] \to R_i\to \widehat{R_i}) =(-1)^{i}1 \otimes (a_i \wedge \ldots \wedge a_\ell),
$$
and Proposition \ref{HAR} gives~(\ref{HAdPl}).

\end{proof}

\section{The Cohomology Ring of the Plactic Monoid Algebra}

\par We will use the notations $\widehat{L_i}[a_1|\ldots|a_\ell] = \widehat{L_i}= (a_1 \wedge (a_2 \vee \cdots \vee a_{i+1})) \cdots (a_{i-1}\wedge (a_i \vee a_{i+1}))( a_i \wedge a_{i+1})( a_{i+2}\cdots a_\ell)$ and $\widehat{R_i}[a_1|\ldots|a_\ell] = \widehat{R_i} = ((a_1\cdots a_{i-1}) (a_i \vee a_{i+1})((a_i \wedge a_{i+1})\vee a_{i+2}) ((a_i \wedge \cdots \wedge a_{\ell-1})) \vee a_n)$. Here $0 \le i \le \ell -1$ and $1 \le j \le \ell$.

\begin{lemma}
Let $M$ be a $\Bbbk \mathrm{Pl}_n$-module and let $\xi \in \mathrm{Hom}_{\Bbbk}(\mathrm{I}\Bbbk, M)$, $\zeta \in \mathrm{Hom}_\Bbbk(\mathfrak{V}^{(\ell-1)}\Bbbk,M)$. Then $\xi \smile \zeta, \zeta \smile \xi \in \mathrm{Hom}_\Bbbk(\mathfrak{V}^{(\ell)}\Bbbk,M)$ can be described
by the formulas
\begin{equation}\label{lm}
(\xi \smile \zeta)[a_1|\ldots|a_{\ell+1}] = \sum\limits_{i=0}^\ell (-1)^{i}\left( \xi[a_1 \vee \cdots \vee a_{i+1}]\varepsilon(\widehat{L_i})\right)\left((a_1 \vee \cdots \vee a_{i+1})\zeta[\widehat{L_i}]\right),
\end{equation}
\begin{equation}\label{rm}
(\zeta \smile \xi)[a_1| \ldots| a_{\ell +1}]= \sum\limits_{j=1}^{\ell+1}(-1)^{\ell+1-j}\left(\zeta[\widehat{R_j}] \varepsilon(a_j \wedge \cdots \wedge a_{\ell +1})\right) \left(\widehat{R_j}\xi[a_j \wedge \cdots \wedge a_{\ell +1}]\right).
\end{equation}
\end{lemma}
\begin{proof}
Indeed, from (\ref{Ag}) follows that we need to find all paths $\{p\}$ of forms,
$$
\mathfrak{V}^{(\ell)} \ni [a_1|\ldots|a_{\ell+1}] \to [b_1|\ldots|b_\ell] \in \mathfrak{V}^{(\ell-1)}
$$
but from construction of $\widehat{R}$, $\widehat{L}$ (see (\ref{Li}), (\ref{Rj})) follows that $\{p\} =  \{\widehat{R},\widehat{L}\}$, and weights of this paths were found in proof of Theorem \ref{CPl}. This completes the proof.
\end{proof}

\smallskip

\begin{lemma}\label{hlm}
Let $B$ be a $\Bbbk\mathrm{Pl}_n^e$-module, and let $\alpha \in \mathrm{Hom}_\Bbbk (\mathrm{I}\Bbbk, B)$, $\beta \in \mathrm{Hom}_\Bbbk(\mathfrak{V}^{(\ell-1)}\Bbbk, B)$. Then $\alpha \smile \beta, \beta \smile \alpha \in \mathrm{Hom}_\Bbbk(\mathfrak{V}^{(\ell)}\Bbbk, B)$ can be described by the formulas
\begin{equation}\label{hlm}
(\xi \smile \zeta)[a_1|\ldots|a_{\ell+1}] = \sum\limits_{i=0}^\ell (-1)^{i}\left( \xi[a_1 \vee \cdots \vee a_{i+1}]\widehat{L_i}\right)\left((a_1 \vee \cdots \vee a_{i+1})\zeta[\widehat{L_i}]\right),
\end{equation}
\begin{equation}\label{hrm}
(\zeta \smile \xi)[a_1| \ldots| a_{\ell +1}]= \sum\limits_{j=1}^{\ell+1}(-1)^{\ell+1-j}\left(\zeta[\widehat{R_j}] (a_j \wedge \cdots \wedge a_{\ell +1})\right) \left(\widehat{R_j}\xi[a_j \wedge \cdots \wedge a_{\ell +1}]\right).
\end{equation}
\end{lemma}

\begin{proof}
The proof is the same as for the previous Lemma with the exception of weights. Arguing as in proof of Theorem \ref{HARPl}, we complete the proof.

\end{proof}

\begin{lemma}\label{eqlemma}
In the above notation, from
$$
\sum\limits_{j=1}^{\ell +1}(-1)^j [\widehat{R_j}] \otimes [a_1 \wedge \cdots \wedge a_{\ell+1}] = \sum\limits_{i=0}^\ell (-1)^{i+1}[\widehat{L_i}] \otimes [a_1 \vee \cdots \vee a_{i+1}],
$$
it follows that $ab=ba$ for all $a,b \in \{a_1, \ldots, a_{\ell+1}\}$.
\end{lemma}

\begin{proof}
Let $j = \ell+1$, then there is $0 \le i \le \ell$ such that $(-1)^{\ell+1}[a_2|\ldots|a_\ell] \otimes [a_1] = (-1)^{i+1}[\widehat{L_i}]\otimes [a_1 \vee \ldots \vee a_{i+1}]$, but it is possible iff $i = \ell$, otherwise there exist at least one element $a_i$ such that $a_i = a_{i+1}$ but it is impossible. Then we get
$$
\begin{cases}
a_1 = a_1 \wedge(a_2 \vee \cdots \vee a_{\ell +1})\\
a_2 = a_2 \wedge (a_3 \vee \cdots \vee a_{\ell +1})\\
\hdotsfor{1}\\
a_\ell = a_\ell \wedge a_{\ell+1}\\
a_{\ell+1} = a_1 \vee \cdots \vee a_{\ell +1}
\end{cases},
$$
from (\ref{comm}) follows $a_\ell a_{\ell+1} = a_{\ell+1}a_\ell$, $a_{\ell-1}a_{\ell} = a_{\ell}a_{\ell-1}$, $\ldots,$ $a_1a_{\ell+1} = a_{\ell+1}a_1$. Let $j = \ell$, then we have $(-1)^{\ell}[a_1|\ldots|a_{\ell-1}|a_\ell \vee a_{\ell+1}] \otimes [a_\ell \wedge a_{\ell+1}] = (-1)^{i-1}[\widehat{L_i}]\otimes [a_1 \vee \ldots \vee a_{i+1}]$ for some $0 \le i \le \ell-1$. We again see it is possible iff $i = \ell-1$, we get

$$
\begin{cases}
a_1 = a_1 \wedge(a_2 \vee \cdots \vee a_{\ell })\\
a_2 = a_2 \wedge (a_3 \vee \cdots \vee a_{\ell })\\
\hdotsfor{1}\\
a_{\ell-1} = a_{\ell-1} \wedge a_\ell\\
a_\ell \vee a_{\ell +1} = a_{\ell+1}\\
a_\ell \wedge a_{\ell+1} = a_1 \vee \cdots \vee a_{\ell}
\end{cases},
$$
using (\ref{comm}) we get $a_{\ell-1} a_{\ell} = a_{\ell}a_{\ell-1}$, $a_{\ell-2}a_{\ell} = a_{\ell}a_{\ell-2}$, $\ldots,$ $a_1a_{\ell} = a_{\ell}a_1$. Using induction on $1 \le j \le \ell -1$ we will obtain that $ab=ba$ for all $a,b \in \{a_1, \ldots, a_{\ell+1}\}$, q.e.d.
\end{proof}

\smallskip

\begin{corollary}\label{smileproduct}
If $ab = ba$ for all $a,b \in \{a_1, \ldots, a_{\ell+1}\}$ then the map $Ag_\bullet: A_\bullet(\Bbbk \mathrm{Pl}_n) \to A_\bullet(\Bbbk \mathrm{Pl}_n) \otimes A_\bullet(\Bbbk \mathrm{Pl}_n)$ can be described by the formulas

\begin{equation*}
Ag_{\ell}[a_1|\ldots|a_{\ell}] =\sum\limits_{p+q =\ell} \sum\limits_{\substack{1 \le  i_1 < \ldots < i_p \le p \\ 1 \le j_1 < \ldots < j_q \le q}}\rho_{PQ} [a_{i_1}|\ldots|a_{i_p}]a_{j_1}\cdots a_{j_q}\otimes a_{i_1}\cdots a_{i_p}[a_{j_1}|\ldots|a_{j_q}]
\end{equation*}
here $\rho_{PQ} = \mathrm{sign}\begin{pmatrix} 1 & \ldots & p& p+1& \ldots &  p+q \\ i_1 & \ldots & i_p &  j_{1} & \ldots & j_{q} \end{pmatrix}.$
\end{corollary}

\begin{proof} Indeed, (\ref{m}) implies that we must find all paths
$$
\mathfrak{V}^{p+q-1} \ni [a_1| \ldots| a_{p+q}] \to [b_1| \ldots |b_{p+q}] \in \left( \Bbbk \mathrm{Pl}_n / \Bbbk\right)^{\otimes(p+q)},
$$
such that $[b_1| \ldots|b_p] \in \mathfrak{V}^{(p-1)}$ and $[b_{p+1}| \ldots|b_{p+q}] \in \mathfrak{V}^{(q-1)}$. Since all $a_1, \ldots, a_{p+q}$ are commutative, all these paths correspond to the permutations $PQ = \begin{pmatrix} 1 & \ldots & p& p+1& \ldots &  p+q \\ i_1 & \ldots & i_p &  j_{1} & \ldots & j_{q} \end{pmatrix}$, and the weights of this paths correspond to the signature of the corresponding permutations, as claimed.
\end{proof}

\smallskip

\begin{theorem}\label{Ext}
Let $\Bbbk\mathrm{Pl}_n$ be the plactic monoid algebra with $n$ generators over the field $\Bbbk$.  The cohomology ring of $\Bbbk\mathrm{Pl}_n$ is isomorphic to the ring
$$
\mathrm{Ext}_{\Bbbk \mathrm{Pl}_n}^*(\Bbbk, \Bbbk) \cong \left.{\bigwedge}_\Bbbk[\xi_1, \ldots, \xi_{n}]\right/(\xi_i\xi_j =0 \mbox{ iff } a_ia_j \ne a_ja_i).
$$
Here $a_i,a_j$ are columns.
\end{theorem}
\begin{proof}
Since there are no relations of the form $ab=\alpha \in \Bbbk$, we may assume that the augmentation map $\varepsilon:\Bbbk\mathrm{Pl}_n \to \Bbbk$ is the identity map, i.e., $\varepsilon(x) = 1$ for any $x \in \mathrm{Pl}_n$. We get the cochain complex
$$
0 \to \Bbbk\mathrm{Pl}_n \xrightarrow{d^0} \mathrm{Hom}_{\Bbbk}(\mathrm{I}\Bbbk, \Bbbk) \xrightarrow{d^1} \mathrm{Hom}_\Bbbk(\mathfrak{V}\Bbbk, \Bbbk) \xrightarrow{d^2} \mathrm{Hom}_\Bbbk(\mathfrak{V}^{(2)}\Bbbk, \Bbbk) \to \ldots
$$
where
$$
(d^0x)[a] = \varepsilon(a) - \varepsilon(a) = 0, \quad (d^1\xi)[a|b] = \xi[a] + \xi[b] - \xi[a \vee b] - \xi[a \wedge b],
$$
and for any $\varphi \in \mathrm{Hom}_\Bbbk(\mathfrak{V}^{(\ell-1)}\Bbbk, \Bbbk)$ we have
$$
(d^\ell\varphi)[a_1|\ldots|a_{\ell+1}] = \sum\limits_{i=0}^\ell(-1)^{i} \varphi[\widehat {L_i}] + \sum\limits_{j=1}^{\ell+1}(-1)^{j} \varphi[ \widehat{R_j}]+  \sum\limits_{m=1}^{\ell-1}\sum\limits_{m+l+k \le \ell-1}(-1)^{l+k}\varphi(W_{m,l,k}).
$$

\par Consider the following functions:
$$
\xi_i(x) = \begin{cases} 1, \mbox{ if } \{e_i\} \subseteq \{x\}, \\ 0, \mbox{ otherwise,}   \end{cases}
$$
here $e_i, \, x = e_{x_1, \ldots, x_\ell} \in \mathrm{I}.$

\par It's not hard to see that $\xi_i(x)$ are cocycles. Indeed, there are following possibilities, 1) let $\{e_i\} \subseteq \{a\}$ and $\{e_i\} \in \{b\}$ then $\{e_i\} \in \{a \vee b\}$ and $\{e_i\} \in \{a \wedge b\}$ it follows $(d^1\xi_i)[a|b] = 1+1-1-1 = 0$, 2) let $\{e_i\} \in \{a\}$ and $\{e_i\} \notin \{b\}$ then $\{e_i\} \in \{a \vee b\}$ and $\{e_i\} \notin \{a \wedge b\}$ it follows $(d^1\xi_i)[a|b] = 1+0-1-0 = 0$, 3) let $\{e_i\} \notin \{a\}$ and $\{e_i\} \in \{b\}$ then we have to consider two cases; 3a) $\{e_i\} \in \{a \vee b\}$ then $\{e_i\} \notin \{a \wedge b\}$ it follows $(d^1\xi_i)[a|b] = 0+1 - 1 - 0 = 0$, 3b) $\{e_i\} \notin \{a \vee b\}$ then $\{e_i\} \in \{a \wedge b\}$ it follows $(d^1\xi_i)[a|b] = 0+1 - 0 - 1 = 0$, 4) let $\{e_i\} \notin \{a\}$ and $\{e_i\} \notin \{b\}$ then $\{e_i\} \notin \{a \vee b\}$ and $\{e_i\} \notin \{a \wedge b\}$ it follows $(d^1\xi_i)[a|b] = 0+0-0-0= 0$.

\par Let $\vartheta_p \in \mathrm{Hom}_\Bbbk(\mathfrak{V}^{(p-1)}\Bbbk, \Bbbk)$ and let $\vartheta_q \in \mathrm{Hom}_\Bbbk(\mathfrak{V}^{(q-1)}\Bbbk, \Bbbk)$. Since the comultiplication $\Bbbk \mathrm{Pl}_n \otimes \Bbbk \mathrm{Pl}_n \leftarrow \Bbbk \mathrm{Pl}_n:\Delta (x) = x \otimes x$ is cocomutative, then the product $\smile$ must be skew commutative, i.e. $(\vartheta_p \smile \vartheta_q)= (-1)^{pq} (\vartheta_q \smile \vartheta_p)$. Consider the commutative diagram
$$
  \xymatrix{
  \mathrm{Hom}_{\Bbbk} (\mathfrak{V}^{(p-1)}\Bbbk, \Bbbk) \otimes \mathrm{Hom}_{\Bbbk} (\mathfrak{V}^{(q-1)}\Bbbk, \Bbbk) \ar@{->}[r]^(.62){\smile} \ar@{->}[d]_{\tau^*} & \mathrm{Hom}_{\Bbbk} (\mathfrak{V}^{(p+q-1)}\Bbbk,\Bbbk) \ar@{->}[d]^{\breve\tau} \\
  \mathrm{Hom}_{\Bbbk} (\mathfrak{V}^{(q-1)}\Bbbk, \Bbbk) \otimes \mathrm{Hom}_{\Bbbk} (\mathfrak{V}^{(p-1)}\Bbbk, \Bbbk) \ar@{->}[r]_(.62){\smile} & \mathrm{Hom}_{\Bbbk} (\mathfrak{V}^{(p+q-1)}\Bbbk, \Bbbk).
  }
$$
Here $\tau:A_\bullet(\Bbbk\mathrm{Pl}_n,\Bbbk) \otimes A_\bullet(\Bbbk\mathrm{Pl}_n,\Bbbk) \to A_\bullet(\Bbbk\mathrm{Pl}_n,\Bbbk) \otimes A_\bullet(\Bbbk\mathrm{Pl}_n,\Bbbk)$ is the chain automorphism such that $\tau: x \otimes y \to (-1)^{\mathrm{deg}(x)\mathrm{deg}(y)}y \otimes x$.

\par We may assume without loss of generality that $p=1$ and $q =\ell$.  Suppose that $\xi \in \mathrm{Hom}_\Bbbk(\mathrm{I}\Bbbk, \Bbbk)$ and $\zeta \in \mathrm{Hom}_\Bbbk(\mathfrak{V}^{(\ell-1)}, \Bbbk)$. Then $\xi \smile \zeta = (-1)^\ell \zeta \smile \xi$. Using (\ref{lm}), we obtain
\begin{multline*}
(\xi \smile \zeta)[a_1|\ldots|a_{\ell+1}] \leftrightsquigarrow  (\xi \bigvee \zeta)\left(\sum\limits_{i=0}^{\ell}(-1)^i[a_1 \vee \cdots \vee a_{i+1}] \otimes [\widehat{L_i}]\right) \xrightarrow{\tau^*} \\ \xrightarrow{\tau^*} (-1)^\ell(\zeta \bigvee \xi)\left(\sum\limits_{i=0}^\ell(-1)^i\widehat{L_i} \otimes [a_1 \vee \cdots \vee a_{i+1}]\right) \leftrightsquigarrow (-1)^\ell (\zeta \smile \xi)[a_1|\ldots|a_{\ell+1}],
\end{multline*}
and (\ref{rm}) implies that
$$
\sum\limits_{j=1}^{\ell +1}(-1)^j [\widehat{R_j}] \otimes [a_1 \wedge \cdots \wedge a_{\ell+1}] = \sum\limits_{i=0}^\ell (-1)^{i+1}[\widehat{L_i}] \otimes [a_1 \vee \cdots \vee a_{i+1}].
$$
Lemma \ref{eqlemma} now gives that if $ab = ba$ for all $a,b \in \{a_1, \ldots, a_{\ell+1}\}$ then $(\xi \smile \zeta)[a_1|\ldots|a_{\ell+1}] \ne 0$. Finally, using  Corollary \ref{smileproduct} and (\ref{lm}), we have
\begin{multline}\label{prodinclas}
(\vartheta_p \smile \vartheta_q)[a_1|\ldots|a_{p+q}] = \\= \begin{cases} \sum\limits_{\substack{1 \le  i_1 < \ldots < i_p \le p \\ 1 \le j_1 < \ldots < j_q \le q}}\rho_{PQ} \vartheta_p[a_{i_1}|\ldots|a_{i_p}]\vartheta_q[a_{j_1}|\ldots|a_{j_q}], \mbox{ iff } a_ia_{i+1} = a_{i+1}a_i \mbox{ for all } 1 \le i \le p+q-1 \\0, \mbox{ otherwise.} \end{cases}
\end{multline}

\smallskip

\par Let us assume now all columns $a_1, \ldots, a_\ell$ are pairwise commutative, then from (\ref{prodinclas}) follows that
\begin{equation}
\xi_{i_1} \smile \cdots \smile \xi_{i_\ell}[a_1| \ldots|a_\ell] = \begin{Vmatrix} \xi_{i_1}(a_1) & \ldots & \xi_{i_1}(a_\ell) \\
 \vdots & \ddots & \vdots \\
 \xi_{i_\ell}(a_1) &\ldots & \xi_{i_\ell}(a_\ell) \end{Vmatrix}
\end{equation}

\par It's not hard to see that if $\xi_{i_k}(a_j) = 0$ then $\xi_{i_k}(a_{j-t}) = 0$ for any $1 \le t \le j-1$, then using Laplace expansion we can express the determinant via determinants of sub-matrices. It means our determinant is non zero iff the correspondence matrix is upper triangular matrix, q.e.d.

\end{proof}

\paragraph{The Hochschild Cohomology Ring $HH^*(\mathrm{Pl}_n)$ of the plactic monoid algebra.} Theorem \ref{HARPl} implies that the Hochschild cohomology of the plactic monoid algebra is isomorphic to the homology of the cochain complex
$$
0 \to \Bbbk\mathrm{Pl}_n \xrightarrow{d^0} \mathrm{Hom}_{\Bbbk}(\mathrm{I}\Bbbk, \Bbbk\mathrm{Pl}_n) \xrightarrow{d^1} \mathrm{Hom}_{\Bbbk}(\mathfrak{V}\Bbbk, \Bbbk\mathrm{Pl}_n) \xrightarrow{d^2} \mathrm{Hom}_{\Bbbk}(\mathfrak{V}^{(2)}\Bbbk, \Bbbk\mathrm{Pl}_n) \to \ldots
$$
Here, for $w \in \Bbbk \mathrm{Pl}_n$, $a,b, a_1, \ldots, a_n \in I$ and $\psi \in \mathrm{Hom}_\Bbbk(\mathfrak{V}^{(\ell-1)},\Bbbk \mathrm{Pl}_n)$, we have
$$
(d^0w)[a] = f(wa) - f(aw), \qquad (d^1\psi^1)[a|b] = \psi[a]b+a\psi[b]-\psi[a \vee b](a \wedge b) - (a \vee b)\psi[a \wedge b],
$$

\begin{multline*}
(d^\ell\psi)[a_1 | \ldots| a_{\ell+1}] = \\ = \sum\limits_{i=0}^{\ell}(-1)^{i}(a_1 \vee \ldots \vee a_{i+1}) \psi[\widehat {L_i}] + \sum\limits_{j=1}^{\ell+1}(-1)^{j} \psi[\widehat{R_j}](a_j \wedge \ldots \wedge a_{m}) +  \sum\limits_{m=1}^{\ell-1}\sum\limits_{m+l+k \le \ell-1}(-1)^{l+k}\psi(W_{m,l,k}).
\end{multline*}

\par As is well known \cite[Lemma 3]{CO} the center $\mathrm{Z} (\mathrm{Pl}_n)$ is equal to the cyclic monoid $\langle e_{1,\ldots,n} \rangle$, i.e., $HH^0(\Bbbk \mathrm{Pl}_n) \cong \Bbbk[e_{1,\ldots,n}]$.

\begin{proposition}\label{der}
For any columns $e_i, a \in \mathrm{I}$, the cochains $\dfrac{\partial a}{\partial e_i}: \mathrm{I}\Bbbk \to \Bbbk\mathrm{Pl}_n$ defined by the rule
$$
\dfrac{\partial a}{\partial e_i} = \begin{cases} \{a \setminus e_i\}, \mbox{ if } \{e_i\} \subseteq \{a\} , \\ 0, \mbox{ otherwise.}  \end{cases}
$$
are derivations, moreover these derivations are not inner.
\end{proposition}

\begin{proof}
\par For any column $c = (c_1; \ldots;c_n)$, put $|c| = c_1 + \ldots + c_n$. Let us prove that $\dfrac{\partial}{\partial e_i} \notin \mathrm{Im}(d^0)$. Assume that $\lambda \vee a = \dfrac{\partial a}{\partial e_i}$ for some $\lambda \in \mathrm{I}$. Then (2.12) implies that $a = \dfrac{\partial a}{\partial e_i} \wedge a$, but then $\left|\dfrac{\partial a}{\partial e_i} \right| \le |a|$ leads to a~contradiction. Suppose now that $a \vee \lambda = \dfrac{\partial a}{\partial e_i}$. Then $|a \vee \lambda| \ge |a|$
gives a~contraction. This means that $\dfrac{\partial }{\partial e_i}\notin \mathrm{Im}(d^0)$.

\smallskip

\par Show that this is a~one-dimensional cocycle. We infer
$$
\left(d^1 \dfrac{\partial}{\partial e_i}\right)[a|b] =  \dfrac{\partial a}{\partial e_i}b + a \dfrac{\partial b}{\partial e_i} - \dfrac{\partial (a \vee b)}{\partial e_i}(a \wedge b) - (a \vee b)\dfrac{\partial (a \wedge b)}{\partial e_i}.
$$
We can represent $\dfrac{\partial y}{\partial e_i}$ as $\left\{\dfrac{\partial y}{\partial e_i}\right\} = \delta_{\{e_i\} \subseteq \{y\}}\{y\} \setminus \{e_i\}$, where
$$
\delta_{\{e_i\} \subseteq \{y\}} = \begin{cases}1, \mbox{ iff } \{e_i\} \subseteq \{y\}, \\ 0, \mbox{ otherwise.}  \end{cases}
$$

\par We obtain:
\begin{multline*}
 \left( d^1 \dfrac{\partial}{\partial x} \right)[a|b]  = \delta_{\{e_i\}\subseteq \{a\}}(\{a\} \setminus \{e_i\})b + \delta_{\{e_i\}\subseteq \{b\}}a(\{b\}\setminus \{e_i\}) - \\ - \delta_{\{e_i\}\subseteq \{a \vee b\}}(\{(a \vee b)\}\setminus \{e_i\})(a \wedge b) - \delta_{\{e_i\} \subseteq \{a \wedge b\}}(a \vee b)(\{a \wedge b\}\setminus \{e_i\}).
\end{multline*}

\par Case 1. $\delta_{\{e_i\}\subseteq \{b\}} = \delta_{\{e_i\} \subseteq \{a \wedge b\}} =0$, $\delta_{\{e_i\}\subseteq \{a\}} = \delta_{\{e_i\}\subseteq \{a \vee b\}} = 1.$ Then $\{(a \setminus e_i) \vee b\} = \{(a \setminus e_i)\} \cup \{b^a\}$, and $\{(a \setminus e_i) \wedge b\} = \{b_a\}$, and also we see that $\{(a \vee b) \setminus e_i\} = \{a \setminus e_i\} \cup \{b^a\}$.

\par Case 2. $\delta_{\{e_i\}\subseteq \{b\}} = \delta_{\{e_i\} \subseteq \{a \wedge b\}} =1$ and $\delta_{\{e_i\}\subseteq \{a\}} = \delta_{\{e_i\}\subseteq \{a \vee b\}} = 0.$ Then $\{(a \wedge b)\setminus e_i\} = \{a \wedge (b \setminus e_i)\}$ and $\{a \vee (b \setminus e_i)\} = \{a\} \cup \{b^a\}$. It is also not hard to see that $\{((a \wedge b)\setminus e_i)^{a \vee b}\} = \{(a \wedge b)\setminus e_i)\} = \{a \wedge (b \setminus e_i)\}$.

\par Case 3. Suppose that $\delta_{\{e_i\} \subseteq \{a \wedge b\}} =0$ and $\delta_{\{e_i\} \subseteq \{b\}}=1$. Then $\delta_{\{e_i\}\subseteq \{a\}} = 0$. We get $\{(a \vee b) \setminus e_i\} = \{a\} \cup \{b^a \setminus e_i\} = \{a\} \cup \{(b \setminus e_i)^a\}$, and $\{a \wedge (b \setminus e_i)\} = \{a \wedge b\}$.

\par Case 4. $\delta_{\{e_i\}\subseteq \{b\}} = \delta_{\{e_i\} \subseteq \{a \wedge b\}} =1$ and $\delta_{\{e_i\}\subseteq \{a\}} = \delta_{\{e_i\}\subseteq \{a \vee b\}} = 1.$ Then $\{e_i\} \subseteq \{a\} \cap \{b\}$. We see that  $\{(a \vee b) \setminus e_i\} = \{a \setminus e_i\} \cup \{b^a\}$ and $\{(a \setminus e_i)\wedge b\} = \{(a \wedge b)\setminus e_i\}$. Therefore, $\{((a \vee b)\setminus e_i)\wedge (a \wedge b)\} = \{((a\setminus e_i)\vee b) \wedge (a \wedge b)\} = \{((a\vee b) \wedge (a \wedge b))\setminus e_i\} = \{(a \wedge b) \setminus e_i\}$. Further, $\{a \wedge (b \setminus e_i)\} = \{(a \wedge b)\setminus e_i\}$ and $\{(a \vee b) \vee ((a \wedge b)\setminus )\} = \{((a \vee b)\vee (a \wedge b))\setminus e_i\} = (a \vee b) \setminus e_i$.

\par If we assume that $\delta_{\{e_i\}\subseteq \{a \vee b\}} = 1$ then $\delta_{\{e_i\}\subseteq \{a\}} = 1$; otherwise, if $\delta_{\{e_i\}\subseteq \{a \vee b\}} = 0$ then $\delta_{\{e_i\}\subseteq \{a\}} = 0$. This means that all possible cases are considered. Moreover any column $a \in \mathrm{I}$ which contains more than one element, i.e., $|c|>1$ can be present as $c = e_i \cdot c'$ where $e_i$ corresponds to the minimal element of column $c$, it means that the cochains $\dfrac{\partial}{\partial e_i}$ are generators of $HH^1(\Bbbk \mathrm{Pl}_n)$.

\end{proof}

\begin{proposition}\label{Prop4.2}
Let $\Bbbk \mathrm{Pl}_n$ be the plactic monoid algebra over a field. Then the multiplication in Hochschild cohomology ring $HH^*(\Bbbk \mathrm{Pl}_n)$ can be described as follows:
\begin{multline}
\left(\dfrac{\partial}{ \partial e_{i_1}} \smile \cdots \smile \dfrac{\partial}{ \partial e_{i_\ell}} \right)[a_1|\ldots|a_\ell] = \\ =  \begin{cases} \sum\limits_{\sigma \in \mathbb{S}_n} \mathrm{sign}(\sigma) (a_\ell \cdots a_1) \dfrac{\partial a_{\sigma(1)}}{ \partial e_{i_1}} \cdots \dfrac{\partial a_{\sigma(\ell)}}{ \partial e_{i_\ell}}, \mbox{ iff } a_ia_j = a_ja_i \mbox{ for all } 1 \le i,j \le \ell,\\ 0, \mbox{ otherwise.}
\end{cases}
\end{multline}
here $\mathrm{sign}(\sigma) = \mathrm{sign}\begin{pmatrix} 1 & \ldots & \ell \\ \sigma(1) & \ldots & \sigma(\ell)  \end{pmatrix}$

\end{proposition}

\begin{proof}
\par Let $\psi_p \in \mathrm{Hom}_\Bbbk(\mathfrak{V}^{(p-1)}\Bbbk, \Bbbk \mathrm{Pl}_n)$ and let $\psi_q \in \mathrm{Hom}_\Bbbk(\mathfrak{V}^{(q-1)}\Bbbk, \Bbbk \mathrm{Pl}_n)$. Since the comultiplication $\Bbbk \mathrm{Pl}_n \otimes \Bbbk \mathrm{Pl}_n \leftarrow \Bbbk \mathrm{Pl}_n:\Delta (x) = x \otimes x$ is cocommutative, the product $\smile$ must be skew commutative; i.e., $(\psi_p \smile \psi_q)= (-1)^{pq} (\psi_q \smile \psi_p)$. Consider the commutative diagram
$$
  \xymatrix{
  \mathrm{Hom}_{\Bbbk} (\mathfrak{V}^{(p-1)}\Bbbk, \Bbbk \mathrm{Pl}_n) \otimes \mathrm{Hom}_{\Bbbk} (\mathfrak{V}^{(q-1)}\Bbbk, \Bbbk \mathrm{Pl}_n) \ar@{->}[r]^(.62){\smile} \ar@{->}[d]_{\tau^*} & \mathrm{Hom}_{\Bbbk} (\mathfrak{V}^{(p+q-1)}\Bbbk,\Bbbk \mathrm{Pl}_n) \ar@{->}[d]^{\breve\tau} \\
  \mathrm{Hom}_{\Bbbk} (\mathfrak{V}^{(q-1)}\Bbbk, \Bbbk \mathrm{Pl}_n) \otimes \mathrm{Hom}_{\Bbbk} (\mathfrak{V}^{(p-1)}\Bbbk, \Bbbk \mathrm{Pl}_n) \ar@{->}[r]_(.62){\smile} & \mathrm{Hom}_{\Bbbk} (\mathfrak{V}^{(p+q-1)}\Bbbk, \Bbbk \mathrm{Pl}_n).
  }
$$
Here $\tau:A_\bullet(\Bbbk\mathrm{Pl}_n,\Bbbk\mathrm{Pl}_n) \otimes A_\bullet(\Bbbk\mathrm{Pl}_n,\Bbbk\mathrm{Pl}_n) \to A_\bullet(\Bbbk\mathrm{Pl}_n,\Bbbk\mathrm{Pl}_n) \otimes A_\bullet(\Bbbk\mathrm{Pl}_n,\Bbbk\mathrm{Pl}_n)$ is the chain automorphism such that $\tau: x \otimes y \to (-1)^{\mathrm{deg}(x)\mathrm{deg}(y)}y \otimes x$.

\par We may assume without loss of generality that $p=1$ and $q =\ell$.  Suppose that $\alpha \in \mathrm{Hom}_\Bbbk(\mathrm{I}\Bbbk, \Bbbk\mathrm{Pl}_n)$ and $\beta \in \mathrm{Hom}_\Bbbk(\mathfrak{V}^{(\ell-1)}, \Bbbk\mathrm{Pl}_n)$. Then we have $\alpha \smile \beta = (-1)^\ell \beta \smile \alpha$. Using (\ref{hlm}), we get
\begin{multline*}
(\alpha \smile \beta)[a_1|\ldots|a_{\ell+1}] \leftrightsquigarrow  (\alpha \bigvee \beta)\left(\sum\limits_{i=0}^\ell(-1)^i\left(\widehat{L_i} \otimes (a_1 \vee \cdots \vee a_{i+1})\right)[a_1 \vee \cdots \vee a_{i+1}] \otimes [\widehat{L_i}]\right) \xrightarrow{\tau^*} \\ \xrightarrow{\tau^*} (-1)^\ell(\beta \bigvee \alpha)\left(\sum\limits_{i=0}^\ell(-1)^i\left(\widehat{L_i} \otimes (a_1 \vee \cdots \vee a_{i+1})\right)[\widehat{L_i}] \otimes [a_1 \vee \cdots \vee a_{i+1}]\right) \leftrightsquigarrow (-1)^\ell (\beta \smile \alpha)[a_1|\ldots|a_{\ell+1}]
\end{multline*}
Now, it follows from~(\ref{hrm}) that
$$
\sum\limits_{j=1}^{\ell +1}(-1)^j \widehat{R_j} \otimes [a_1 \wedge \cdots \wedge a_{\ell+1}] = \sum\limits_{i=0}^\ell (-1)^{i+1}\widehat{L_i} \otimes [a_1 \vee \cdots \vee a_{i+1}].
$$
Using Lemma \ref{eqlemma}, we obtain that if $ab = ba$ for all $a,b \in \{a_1, \ldots, a_{\ell+1}\}$ then $(\alpha \smile \beta)[a_1|\ldots|a_{\ell+1}] \ne 0$. Finally, making use of Corollary \ref{smileproduct} and (\ref{hlm}), we obtain
\begin{multline*}
(\psi_p \smile \psi_q)[a_1|\ldots|a_{p+q}] = \\= \begin{cases} \sum\limits_{\substack{1 \le  i_1 < \ldots < i_p \le p \\ 1 \le j_1 < \ldots < j_q \le q}}\rho_{PQ} \psi_p[a_{i_1}|\ldots|a_{i_p}](a_1 \cdots a_{p+q})\psi_q[a_{j_1}|\ldots|a_{j_q}] \mbox{ if } a_ia_{i+1} = a_{i+1}a_i \mbox{ for all } 1 \le i \le p+q-1 \\0, \mbox{ otherwise.} \end{cases}
\end{multline*}

\par This means that
$$
\dfrac{\partial^2}{\partial x \partial y}[a|b]:= \left( \dfrac{\partial}{\partial x} \smile \dfrac{\partial}{\partial y} \right)[a|b] = \left(\dfrac{\partial a}{\partial x}b\right) \left( a \dfrac{\partial b}{\partial  y}\right) - \left(\dfrac{\partial (a \vee b)}{\partial x}(a \wedge b)\right) \left( (a \vee b) \dfrac{\partial (a \wedge b)}{\partial  y}\right),
$$
and, more generally,
\begin{multline*}
\dfrac{\partial^{p+q}}{\partial x_1 \cdots \partial x_p \partial y_1 \cdots \partial y_q}[a_1|\ldots| a_{p+q}] :=  \left( \dfrac{\partial^p}{\partial x_1 \cdots \partial x_p} \smile  \dfrac{\partial^q}{\partial y_1 \cdots \partial y_q} \right)[a_1|\ldots |a_{p+q}] =  \\ = \begin{cases} \sum\limits_{\substack{1 \le  i_1 < \ldots < i_p \le p \\ 1 \le j_1 < \ldots < j_q \le q}}\rho_{PQ}\left(\dfrac{\partial^p}{\partial x_1 \cdots \partial x_p}[a_{i_1}| \ldots |a_{i_p}]\right)(a_1\cdots a_{p+q})\left(\dfrac{\partial^q}{\partial y_1 \cdots \partial y_q}[a_{j_1}| \ldots|a_{j_q}] \right), \mbox{ iff } a_i a_j = a_ja_i, \\ 0, \mbox{ otherwise,} \end{cases}
\end{multline*}
where $\rho_{PQ} = \mathrm{sign}\begin{pmatrix} 1 & \ldots & p& p+1& \ldots &  p+q \\ i_1 & \ldots & i_p &  j_{1} & \ldots & j_{q} \end{pmatrix},$ $1 \le i<j \le p+q.$ This completes the proof.
\end{proof}

\begin{lemma}\label{Lemma4.4}
Let us suppose that all columns $a_1, \ldots, a_\ell$ are commutative, and let us suppose that for any $1 \le i,j \le \ell$ we have $\dfrac{\partial a_i}{\partial e_j} \ne 0$, then $\left( \dfrac{\partial}{ \partial e_1} \smile \cdots \smile \dfrac{\partial}{ \partial e_\ell}\right)[a_1| \ldots|a_\ell] = 0$.
\end{lemma}
\begin{proof}
First of all let us prove the following formulae,
\begin{equation}\label{4.27}
\dfrac{\partial a}{ \partial e_i} \dfrac{\partial b}{ \partial e_j} = \begin{cases} \dfrac{\partial b}{ \partial e_i} \dfrac{\partial a}{ \partial e_j}, \mbox{ if } \{e_j\} \subseteq \{a\} \\ \\ \dfrac{\partial b}{ \partial e_j} \dfrac{\partial a}{ \partial e_i}, \mbox{ if } \{e_j\} \subseteq \{b  \setminus a\},   \end{cases} \, i \ne j
\end{equation}
Indeed, if $\{e_j\} \subseteq \{a\}$ then
$$
\dfrac{\partial a}{ \partial e_i} \vee \dfrac{\partial b}{ \partial e_j} = \{a \setminus e_i\} \cup \{b \setminus a\} \cup (\{a  \setminus e_j\})^{\{a \setminus e_i\}} = \{a \setminus e_i\} \cup \{b \setminus a\}  = \dfrac{\partial b}{ \partial e_i},
$$
and from (\ref{comm}) follows $\dfrac{\partial a}{ \partial e_i} \wedge \dfrac{\partial b}{ \partial e_j} = \dfrac{\partial a}{ \partial e_j}.$ Let $\{e_j \} \nsubseteq \{a\}$, then
\begin{multline*}
\dfrac{\partial a}{ \partial e_i} \vee \dfrac{\partial b}{ \partial e_j} =\{a \setminus e_i\} \cup ((\{b \setminus a\} \setminus \{e_j\}) \cup \{a\})^{\{a\setminus e_i\}} = (\{a \setminus e_i\}) \cup((\{b \setminus a\} \setminus \{e_j\})\cup \{e_i\} = \\ = \{a\} \cup ((\{b \setminus a\} \setminus \{e_j\}) = \{b  \setminus e_j\} = \dfrac{\partial b}{ \partial e_j},
\end{multline*}
and from (\ref{comm}) follows $\dfrac{\partial a}{ \partial e_i} \wedge \dfrac{\partial b}{ \partial e_j} = \dfrac{\partial a}{\partial e_i}$, as claimed.

\par Now, let us consider the sum
$$
\left(\dfrac{\partial}{ \partial e_{i_1}} \smile \cdots \smile \dfrac{\partial}{ \partial e_{i_\ell}} \right)[a_1|\ldots|a_\ell] =  \sum\limits_{\sigma \in \mathbb{S}_n} \mathrm{sign}(\sigma) (a_{\ell} \cdots a_1) \dfrac{\partial a_{\sigma(1)}}{ \partial e_{i_1}} \cdots \dfrac{\partial a_{\sigma(\ell)}}{ \partial e_{i_\ell}},
$$
from condition $\dfrac{\partial a_i}{\partial e_j} \ne 0$ for any $1 \le i,j \le \ell$ follows $\{e_1\}, \ldots, \{e_\ell\} \subseteq \{a_1\}$, then using (\ref{4.27}) we see that for any permutations $\sigma, \pi$ we get $\dfrac{\partial a_{\sigma(1)}}{ \partial e_{i_1}} \cdots \dfrac{\partial a_{\sigma(\ell)}}{ \partial e_{i_\ell}} = \dfrac{\partial a_{\pi(1)}}{ \partial e_{i_1}} \cdots \dfrac{\partial a_{\pi(\ell)}}{ \partial e_{i_\ell}}$, i.e., $\left( \dfrac{\partial}{ \partial e_1} \smile \cdots \smile \dfrac{\partial}{ \partial e_\ell}\right)[a_1| \ldots|a_\ell] = 0$. q.e.d.
\end{proof}

\begin{theorem}
For the plactic monoid algebra $\Bbbk [\mathrm{Pl}_n]$ the Hochschild cohomology algebra can be describe as below
$$
HH^*(\Bbbk[\mathrm{Pl}_n]) \cong \left.{\bigwedge}_\Bbbk[\alpha_1, \ldots, \alpha_n]\right/(\alpha_i\alpha_j = 0, \mbox{ iff } a_ia_j \ne a_ja_i \mbox{ here } a_i,a_j \in \mathrm{I})
$$
\end{theorem}
\begin{proof}
From Proposition \ref{Prop4.2} follows that it is enough to prove there are not another relations except skew commutativity. First of all let us remark from (\ref{der}) follows the following property,
\begin{equation}\label{inert}
\mbox{if } \dfrac{\partial a_j}{\partial e_{i_k}} = 0, \mbox{then } \dfrac{\partial a_{j-t}}{\partial e_{i_k}} = 0 \mbox{ for any } 1 \le t \le j.
\end{equation}
\par Let us consider the sum
$$
\left(\dfrac{\partial}{ \partial e_{i_1}} \smile \cdots \smile \dfrac{\partial}{ \partial e_{i_\ell}} \right)[a_1|\ldots|a_\ell] =  \sum\limits_{\sigma \in \mathbb{S}_n} \mathrm{sign}(\sigma) (a_{\ell} \cdots a_1) \dfrac{\partial a_{\sigma(1)}}{ \partial e_{i_1}} \cdots \dfrac{\partial a_{\sigma(k)}}{ \partial e_{i_k}} \cdots  \dfrac{\partial a_{\sigma(\ell)}}{ \partial e_{i_\ell}},
$$
let us assume that $\dfrac{\partial a_j}{\partial e_{i_j}} = 0$. Consider now the set $S_{jk}: = \{\sigma \in \mathbb{S}_n: \sigma(k) = j-t, \, 1 \le t \le j-1\}$, then we have
$$
\sum\limits_{\sigma \in \mathbb{S}_n} \mathrm{sign}(\sigma) (a_{\ell} \cdots a_1) \dfrac{\partial a_{\sigma(1)}}{ \partial e_{i_1}} \cdots \dfrac{\partial a_{\sigma(k)}}{ \partial e_{i_k}} \cdots  \dfrac{\partial a_{\sigma(\ell)}}{ \partial e_{i_\ell}} = \sum\limits_{\sigma \in \mathbb{S}_n \setminus S_{jk}} \mathrm{sign}(\sigma) (a_{\ell} \cdots a_1) \dfrac{\partial a_{\sigma(1)}}{ \partial e_{i_1}} \cdots \dfrac{\partial a_{\sigma(k)}}{ \partial e_{i_k}} \cdots  \dfrac{\partial a_{\sigma(\ell)}}{ \partial e_{i_\ell}},
$$
We see that after repeating this procedure we'll get
$$
\sum\limits_{\sigma \in \mathbb{S}_n} \mathrm{sign}(\sigma) (a_{\ell} \cdots a_1) \dfrac{\partial a_{\sigma(1)}}{ \partial e_{i_1}} \cdots \dfrac{\partial a_{\sigma(k)}}{ \partial e_{i_k}} \cdots  \dfrac{\partial a_{\sigma(\ell)}}{ \partial e_{i_\ell}} = \sum\limits_{\sigma \in X} \mathrm{sign}(\sigma) (a_{\ell} \cdots a_1) \dfrac{\partial a_{\sigma(1)}}{ \partial e_{i_1}} \cdots \dfrac{\partial a_{\sigma(k)}}{ \partial e_{i_k}} \cdots  \dfrac{\partial a_{\sigma(\ell)}}{ \partial e_{i_\ell}},
$$
from Lemma \ref{Lemma4.4} follows this sum is non zero iff $X$ consists only one permutation, q.e.d.
\end{proof}

\begin{remark}
The rings of the form $\left.{\bigwedge}_R[\alpha_1, \ldots, \alpha_m]\right/(\alpha_i \alpha_j = 0 \mbox{ iff }i, j \in J)$, where $J$ is a set, were considered in \cite{RVP}, where there were investigated
diagrams associated with a finite simplicial complex in various algebraic and topological categories. These rings are called the Stanley --- Reisner rings or rings of faces.
\end{remark}

\paragraph{Acknowledgements.} The author would like to express his deepest gratitude to Professor Leonid A. Bokut', who familiarized me with the notion of the plactic monoid. He also thanks all members of the Seminar of the Center of Combinatorial Algebra of South China Normal University
and especially Professor Yuqun Chen for useful remarks. Special thanks are due to Doctor Todor Popov for very useful discussions and for having kindly clarified some very important details. The author is also extremely indebted to Professor Patrick Dehornoy, who explained to me connections with knot theory and braids groups.

\end{document}